\documentclass[11pt]{article}
\usepackage{epigamath}

\setpapertype{A4}

\usepackage[english]{babel}

\usepackage[dvips]{graphicx}     

\title{\vspace{-1cm} Syzygies of Prym and paracanonical curves of genus $8$}
\titlemark{Syzygies of Prym and paracanonical curves of genus $8$}
\author{Elisabetta Colombo, Gavril Farkas, Alessandro Verra, and Claire Voisin}
\authoraddresses{
\authordata{Elisabetta Colombo}{\firstname{Elisabetta} \lastname{Colombo}\\
\institution{Universit\`a di Milano, Dipartimento di Matematica, Via Cesare Saldini 50,
 20133 Milano, Italy}\\
\email{elisabetta.colombo@unimi.it}}

\authordata{Gavril Farkas}{\firstname{Gavril} \lastname{Farkas}\\
\institution{Humboldt-Universit\"at zu Berlin,
Institut f\"ur Mathematik, \indent Unter den Linden 6,
10099 Berlin, Germany}\\
\email{farkas@math.hu-berlin.de}}

\authordata{Alessandro Verra}{\firstname{Alessandro} \lastname{Verra}\\
\institution{Universit\`a Roma Tre, Dipartimento di Matematica, Largo San Leonardo Murialdo, 1-00146 Roma, Italy}\\
\email{verra@mat.uniroma3.it}}

\authordata{Claire Voisin}{\firstname{Claire} \lastname{Voisin}\\
\institution{Coll\'ege de France 3, Rue d'Ulm, 75005 Paris, France}\\
\email{claire.voisin@imj-prg.fr}}
}
\authormark{E. Colombo, G. Farkas, A. Verra \& C. Voisin}
\date{\vspace{-5ex}} 
\journal{\'Epijournal de G\'eom\'etrie Alg\'ebrique} 
\acceptation{Received by the Editors on December 15, 2016, and in final form on June 23, 2017. \\ Accepted on June 27, 2017.}

\newcommand{\articlereference}{Volume 1 (2017), Article Nr. 7}

\newtheorem{theo}{Theorem}
\newtheorem{prop}[theo]{Proposition}
\newtheorem{lemm}[theo]{Lemma}
\newtheorem{coro}[theo]{Corollary}
\newtheorem{rema}[theo]{Remark}
\newtheorem{defi}[theo]{Definition}

\newtheorem{sublemm}[theo]{Sublemma}
\newtheorem{question}[theo]{Question}

\def\PP{{\textbf P}}
\def\OO{\mathcal{O}}
\def\cN{\mathcal{N}}

\def\cA{\mathcal{A}}

\def\G{\mathcal{G}}

\def\L{\mathcal{L}}
\def\I{\mathcal{I}}

\def\cM{\mathcal{M}}
\def\cR{\mathcal{R}}

\def\rr{\overline{\mathcal{R}}}
\def\cZ{\mathcal{Z}}

\def\cC{\mathcal{C}}
\def\H{\mathcal{H}}
\def\Pic0{{\rm Pic}^0(X)}

\def\mm{\overline{\mathcal{M}}}

\def\pr{\widetilde{\mathcal{R}}}

\begin{document}


\maketitle



\begin{prelims}

\vspace{-0.15cm}

\def\abstractname{Abstract}
\abstract{By analogy with Green's Conjecture on syzygies of canonical curves, the Prym-Green conjecture predicts that the resolution of a general level $p$ paracanonical curve of genus $g$ is natural. The Prym-Green Conjecture is known to hold in odd genus for almost all levels. Probabilistic arguments strongly suggested that the conjecture might fail for level 2 and genus 8 or 16. In this paper, we present three geometric proofs of the surprising failure of the Prym-Green Conjecture in genus 8, hoping that the methods introduced here will shed light on all the exceptions to the Prym-Green Conjecture for genera with high divisibility by~2.}



\vspace{0.05cm}

\languagesection{Fran\c{c}ais}{%

\textbf{Titre. Syzygies de Prym et courbes paracanoniques de genre 8}
\commentskip
\textbf{R\'esum\'e.}
Par analogie avec la conjecture de Green sur les syzygies des courbes canoniques, la conjecture de Prym-Green pr\'edit que la r\'esolution d'une courbe g\'en\'erale, paracanonique, de genre $g$ et de niveau $p$ est naturelle. Cette conjecture est connue en genre impair pour presque tout niveau. Des arguments probabilistes ont fortement sugg\'er\'e qu'elle pourrait s'av\'erer fausse pour le niveau 2 en genre 8 et 16. Dans cet article, nous pr\'esentons trois d\'emonstrations g\'eom\'etriques de la surprenante non-validit\'e de la conjecture de Prym-Green en genre 8, en esp\'erant que les m\'ethodes introduites apporteront un \'eclairage nouveau sur toutes les exceptions \`a la conjecture de Prym-Green pour des genres divisibles par une grande puissance de 2.}

\end{prelims}


\newpage

\setcounter{tocdepth}{1}
\tableofcontents

\section{Introduction}

By analogy with Green's Conjecture on the syzygies of a general canonical curve \cite{V1}, \cite{V2}, the Prym-Green Conjecture, formulated in \cite{FL} and \cite{chiodetal},
predicts that the resolution of a \emph{paracanonical} curve
$$\phi_{K_C\otimes \eta}:C\hookrightarrow \PP^{g-2},$$
where $C$ is a general curve of genus $g$ and $\eta\in \mbox{Pic}^0(C)[\ell]$ is an $\ell$-torsion point is natural. For even genus $g=2i+6$, the Prym-Green Conjecture amounts to the vanishing statement
\begin{equation}\label{pg1}
K_{i,2}(C,K_C\otimes \eta)=K_{i+1,1}(C,K_C\otimes \eta)=0,
\end{equation}
in terms of Koszul cohomology groups. Equivalently, the genus $g$ paracanonical level $\ell$ curve $C\subseteq \PP^{g-2}$ satisfies the Green-Lazarsfeld property $(N_i)$. The Prym-Green Conjecture has been proved for all \emph{odd} genera $g$ when $\ell=2$, see \cite{FK1}, or $\ell \geq \sqrt{\frac{g+2}{2}}$, see \cite{FK2}. For even genus, the Prym-Green Conjecture has been established by degeneration and using computer algebra tools in \cite{chiodetal} and \cite{cofre}, for all $\ell\leq 5$ and $g\leq 18$, with two possible mysterious exceptions in level $2$ and genus $g=8, 16$ respectively. The last section of \cite{chiodetal} provides various pieces of evidence, including a probabilistic argument, strongly suggesting that for  $g=8$, one has $\mbox{dim } K_{1,2}(C,K_C\otimes \eta)=1$, and thus the vanishing (\ref{pg1}) fails in this case. It is tempting to believe that the exceptions $g=8,16$ can be extrapolated to higher genus, and that for genera $g$ with high divisibility by $2$, there are genuinely novel ways of constructing syzygies of Prym-canonical curves waiting to be discovered.  It would be very interesting to test experimentally the next relevant case $g=24$. Unfortunately, due to memory and running time constraints, this is currently completely out of reach, see \cite{chiodetal} and \cite{es}.

\vskip 4pt

The aim of this paper is to confirm the expectation formulated in \cite{chiodetal} and offer several geometric explanations for the surprising failure of the Prym-Green Conjecture in genus $8$, hoping that the geometric methods described here for constructing syzygies of Prym-canonical curves will eventually shed light on all the exceptions to the Prym-Green Conjecture. We choose a general Prym-canonical curve of genus $8$
$$\phi_{K_C\otimes \eta}:C\hookrightarrow \PP^6,$$
with $\eta^{\otimes 2}=\OO_C$.
Set $L:=K_C\otimes \eta$ and  denote $I_{C,L}(k):=\mbox{Ker}\bigl\{\mbox{Sym}^k H^0(C,L)\rightarrow H^0(C,L^{\otimes k})\bigr\}$ for all $k\geq 2$. Observe that $\mbox{dim } I_{C,L}(2)=\mbox{dim } K_{1,1}(C,L)=7$ and $\mbox{dim } I_{C,L}(3)=49$, therefore as $[C,\eta]$ varies in moduli, the multiplication map
$$\mu_{C,L}: I_{C,L}(2)\otimes H^0(C,L)\rightarrow I_{C,L}(3)$$ globalizes to a morphism of vector bundles of the same rank over the stack $\cR_{8}$ classifying pairs $[C,\eta]$, where
$C$ is a smooth curve of genus $8$ and $\eta \in \mbox{Pic}^0[2]\setminus \{\OO_C\}$.

\begin{theo}\label{theomainintro} For a general Prym curve $[C, \eta]\in \cR_{8}$, one has  $K_{1,2}(C,L)\not=0$. Equivalently
the multiplication map $\mu_{C,L}:I_{C,L}(2)\otimes H^0(C,L)\rightarrow I_{C,L}(3)$ is not an isomorphism.
\end{theo}


\vskip 4pt
We present three different proofs of Theorem \ref{theomainintro}.
The first proof, presented in Section \hyperref[secthirdproof]{3} uses the structure theorem already pointed out in \cite{chiodetal}
for degenerate syzygies of paracanonical curves in $\PP^6$. Precisely, if a paracanonical genus $8$ curve
$\phi_{K_C\otimes \eta}:C\hookrightarrow \PP^6$, where $\eta\neq \OO_C$, has a syzygy $0\neq \gamma \in K_{1,2}(C,K_C\otimes \eta)$
of \emph{sub-maximal rank} (see Section \hyperref[sec1]{2} for a precise definition), then the syzygy scheme of $\gamma$ consists of an isolated point $p\in \PP^6 \setminus C$
and a residual septic elliptic curve $E\subseteq \PP^6$ meeting $C$ transversally along a divisor $e$ of degree $14$, such that if $e$ is viewed as a divisor on $C$ and $E$ respectively, then
\begin{equation}\label{cond2}
e_C\in |K_C\otimes \eta^{\otimes 2}| \ \ \mbox{ and }\ \  e_E\in |\OO_E(2)|.
\end{equation}

The union $D:=C\cup E\hookrightarrow \PP^6$, endowed with the line bundle $\OO_D(1)$ is a degenerate spin curve of genus $22$ in the sense of \cite{Cor}. The locus of stable spin structures with at least $7$ sections defines a subvariety of codimension $21={7\choose 2}$ inside the moduli space $\overline{\mathcal{S}}_{22}^{-}$ of stable odd spin curves of genus $22$. By restricting this condition to the locus of spin structures having $D:=C\cup_e E$ as underlying curve, it turns out that one has enough parameters to realize this condition for a general $C\subseteq \PP^6$ if and only if
$$\mbox{dim} |K_C\otimes \eta^{\otimes 2}|=7,$$
which happens precisely when $\eta^{\otimes 2}\cong \OO_C$. Therefore for each Prym-canonical curve $C\subseteq \PP^6$ of genus $8$ there exists a corresponding elliptic curve $E\subseteq \PP^6$ such that the intersection divisor $E\cdot C$ verifies (\ref{cond2}), which forces $K_{1,2}(C,K_C\otimes \eta)\neq 0$.

\vskip 4pt

The second and the third proofs involve the reformulation given in Section \hyperref[subsecquartic]{2.B} (see
Proposition \ref{proquarticsing}) of the condition that a paracanonical curve $\phi_L:C\hookrightarrow \PP^6$  have a non-trivial syzygy.
Precisely, if $\phi_L(C)$ is scheme-theoretically generated by quadrics, then $K_{1,2}(C,L)\neq 0$, if and only if  there exists a quartic hypersurface in $\PP^6$ singular along $C\subseteq \PP^6$, which is  not a quadratic polynomial in quadrics vanishing along $C$, that is, it does not belong to the image of the multiplication map
$$\mbox{Sym}^2 I_{C,L}(2)\rightarrow I_{C,L}(4).$$
Equivalently, one has $H^1(\PP^6, \mathcal{I}_{C/\PP^6}^2(4))\neq 0$.

\vskip 5pt

The second proof presented in Section \hyperref[secfirstproof]{4} uses intersection theory on the stack $\rr_8$.  The virtual Koszul divisor
of Prym curves $[C,\eta]\in \cR_8$ having $K_{1,2}(C,K_C\otimes \eta)\neq 0$, splits into two divisors $\mathfrak{D}_1$ and $\mathfrak{D}_2$ respectively,  corresponding
to the case whether $C\subseteq \PP^6$ is not scheme-theoretically cut out by quadrics, or $H^1(\PP^6, \mathcal{I}_{C/\PP^6}^2(4))\neq 0$ respectively.
We determine the virtual classes of both closures $\overline{\mathfrak{D}}_1$ and $\overline{\mathfrak{D}}_2$. Using an explicit uniruled parametrization of $\rr_8$
constructed in \cite{FV}, we conclude that the class $[\overline{\mathfrak{D}}_2]\in CH^1(\rr_8)$ cannot possibly be effective (see Theorem \ref{theomain}). Therefore, again $K_{2,1}(C,K_C\otimes \eta)\neq 0$, for every
Prym curve $[C, \eta]\in \cR_8$.

\vskip 5pt

The third proof given in Section \hyperref[secsecondproof]{5} even though subject to a plausible, but still unproved transversality assumption, is constructive and potentially the most useful, for we feel it might offer hints to the case $g=16$ and further. The idea is to consider rank $2$ vector bundles $E$ on $C$ with canonical determinant and $h^0(C,E)=h^0(C, E(\eta))=4$. (Note that
 the condition that $\eta$ is $2$-torsion is equivalent to the fact
 that $E(\eta)$ also has canonical determinant, which is essential  for the existence of
 such nonsplit vector bundles, cf. \cite{mumford}.) By pulling back to $C$ the determinantal quartic hypersurface consisting of rank $3$ tensors in $$\PP\Bigl (H^0(C,E)^{\vee}\otimes H^0(C,E(\eta))^{\vee}\Bigr)\cong \PP^{15}$$
under the natural map $H^0(C,K_C\otimes \eta)^{\vee}\rightarrow H^0(C,E)^{\vee}\otimes H^0(C,E(\eta))^{\vee}$, we obtain explicit
quartic hypersurfaces singular along the curve $C\subseteq \PP^6$.  Our proof
that these are not quadratic polynomials into quadrics vanishing along the curve, that is, they do not lie in the image of $\mbox{Sym}^2 I_{C,L}(2)$ remains incomplete,
but there is a lot of  evidence for this.

\vskip 3pt
The methods of Section \hyperref[secsecondproof]{5} suggests the following analogy in the next  case $g=16$. If $[C,\eta]\in \cR_{16}$ is a Prym curve of genus $16$, there exist vector bundles $E$ on $C$ with $\mbox{det } E\cong K_C$ and satisfying $h^0(C, E)=h^0(C, E(\eta))=6$. Potentially they could be used to prove that $K_{5,2}(C,K_C\otimes \eta)\not=0$ and thus confirm the next exception to the Prym-Green Conjecture.

\section{Syzygies of paracanonical curves of genus 8}\label{sec1}

Let $C$ be a general smooth projective curve of genus $8$.  For a non-trivial line bundle  $\eta\in \mbox{Pic}^0(C)$,
we shall study the \emph{paracanonical} line bundle $L:=K_C\otimes \eta$. When $\eta$ is a $2$-torsion point, we speak of the
\emph{Prym-canonical} line bundle $L$. For each paracanonical bundle $L$,
we have $h^0(C,L)=7$ and an induced embedding
$$\phi_L: C\hookrightarrow \PP^6.$$  The goal is to understand
the reasons for the non-vanishing of the Koszul group $K_{1,2}(C,L)$ of a Prym-canonical bundle $L$, as suggested experimentally by
the results of \cite{chiodetal}, \cite{cofre}.

\vskip 4pt

Let $I_C(2)=I_{C,L}(2)\subseteq  H^0(\PP^6,\mathcal{O}_{\PP^6}(2))$, respectively $I_C(3)=I_{C,L}(3)\subseteq  H^0(\PP^6,\mathcal{O}_{\PP^6}(3))$ be the ideal of quadrics, respectively cubics, vanishing
on $\phi_L(C)$. It is well-known that whenever $L$ is projectively normal, the non-vanishing of the Koszul cohomology group $K_{1,2}(C,L)$ is equivalent to the non-surjectivity of the multiplication map
\begin{eqnarray}\label{eqmulti} \mu_{C,L}: H^0(\PP^6,\OO_{\PP^6}(1))\otimes I_C(2)\rightarrow I_C(3).
\end{eqnarray}

Note that
$$\mbox{ dim }\,I_C(2)={8 \choose 2}-21=7,\ \ \mbox{ and } \ \ {\rm dim}\,I_C(3)={9\choose 3}-3\cdot 14+7=49,$$
respectively, so that the two spaces appearing in the map (\ref{eqmulti}) have the same dimension. Denote by $P^{14}_8$ the universal degree $14$ Picard variety over
$\cM_8$ consisting of pairs $[C,L]$, where $[C]\in \cM_8$ and $L\neq K_C$. The jumping locus
$$\mathfrak{Kosz}:=\Bigl\{[C,L]\in P^{14}_8: K_{1,2}(C,L)\neq 0\Bigr\}$$
is a divisor. It turns out, cf. Theorem 5.3 of \cite{chiodetal} and Proposition \ref{splitting}, that $\mathfrak{Kosz}$ splits into two components depending on the \emph{rank} of the corresponding non-zero syzygy from $K_{1,2}(C,L)$.

\vskip 3pt
\begin{defi}{\rm{
The rank of a non-zero syzygy  $\gamma=\sum_{i=0}^6 \ell_i\otimes q_i\in \mbox{Ker}(\mu_{C,L})$ is the dimension of the subspace $\langle \ell_0, \ldots, \ell_6\rangle \subseteq H^0(\PP^6, \OO_{\PP^6}(1))$.
The syzygy scheme $\mbox{Syz}(\gamma)$ of $\gamma$ is the largest subscheme $Y\subseteq \PP^6$ such that $\gamma\in H^0(\PP^6, \OO_{\PP^6}(1))\otimes I_Y(2)$.}}
\end{defi}

It is shown in \cite{chiodetal}, that $\mathfrak{Kosz}$ splits into divisors $\mathfrak{Kosz}_6$ and $\mathfrak{Kosz}_7$, depending on whether the syzygy $0\neq \gamma \in \mbox{Ker}(\mu_{C,L})$ has rank $6$ or $7$ respectively. By a specialization argument to irreducible nodal curves, it follows from \cite{chiodetal} that $\cR_{8}\nsubseteq \mathfrak{Kosz}_7$. A direct, more transparent proof of this fact will be given in Proposition \ref{prodeg}.

\subsection{Paracanonical curves of genus $8$ with special syzygies and elliptic curves}
\label{subsecelliptic}

We summarize a few facts already stated or recalled in Section 5 of \cite{chiodetal} concerning rank $6$ syzygies of paracanonical curves in $\PP^6$. Very generally, let
$$\gamma= \sum_{i=1}^6 \ell_i\otimes q_i\in H^0(\PP^6,\OO_{\PP^6}(1))\otimes H^0(\PP^6,\OO_{\PP^6}(2))$$ be a rank 6 linear syzygy among quadrics in $\PP^6$.  The linear forms $\ell_1, \ldots, \ell_6$ define a  point $p\in \PP^6$. Following Lemma 6.3 of \cite{Sc}, there exists
a skew-symmetric matrix of linear forms $A:=(a_{ij})_{i,j=1, \ldots, 6}$, such that  $$q_i=\sum_{j=1}^6 \ell_j a_{ij}.$$
In the space $\PP^{20}$ with coordinates $\ell_1, \ldots, \ell_6$ and $a_{ij}$ for $1\leq i<j\leq 6$, one considers the $15$-dimensional variety $X_6$  defined by the $6$ quadratic equations $\sum_{j=1}^6 \ell_j a_{ij}=0$, where $i=1, \ldots, 6$ and by the cubic equation $\mbox{Pfaff}(A)=0$ in the variables $a_{ij}$.
The original space $\PP^6$ embeds in $\PP^{20}$ via evaluation. The syzygy scheme $\mbox{Syz}(\gamma)$ is the union of the point $p$ and of the intersection $D$ of $\PP^6$ with the variety $X_6$. It follows from Theorem 4.4 of \cite{ELMS}, that for a general rank $6$ syzygy $\gamma$ as above, $D\subseteq \PP^6$ is a smooth curve of genus $22$  and degree $21$ such that $\OO_D(1)$ is a theta characteristic.

\vskip 4pt

In the case at hand, that is, when $[C,L]\in \mathfrak{Kosz}_6$, the curve $D$ must be reducible, for it has $C$ as a component. More precisely:

\begin{lemm}\label{DCE} For a general paracanonical curve $C\subseteq \PP^6$ having a rank $6$ syzygy, the curve $D$ is nodal and consists of two components $C\cup E$, where
$E\subseteq \PP^6$ is an elliptic septic curve. Furthermore,  $\OO_D(2)=\omega_D$.
The intersection  $e:=C\cdot E$, viewed as a divisor on $C$ satisfies $e_C\in |\OO_C(2)\otimes K_C^{\vee}|$, and as a divisor on $E$, satisfies $e_E\in |\OO_E(2)|$.
\end{lemm}


\begin{rema} {\rm Note that $C$ is Prym-canonical or canonical if and only if $e_C\in |K_C|$.}
\end{rema}

The construction above is reversible. Firstly, general element $[C,L]\in \mathfrak{Kosz}_6$ can be reconstructed as the residual curve of a reducible spin curve $D\subseteq \PP^6$ of genus $22$  containing an elliptic curve $E\subseteq \PP^6$ with $\mbox{deg}(E)=7$ as a component such that the union of $D$ and some point $p\in \PP^6\setminus E$ is the syzygy scheme of a rank 6 linear syzygy among quadrics in $\PP^6$.

\vskip 4pt

Furthermore, given a reducible spin curve $D=C\cup_e E\subseteq \PP^6$ of genus $22$ as above, that is, with $\omega_D\cong \OO_D(2)$,
 the genus $8$  component $C$ has a nontrivial syzygy of rank $6$ involving the quadrics in the $6$-dimensional subspace $I_{D}(2)\subseteq I_C(2)$, see Lemma \ref{ID} for a proof
 of this fact.

\subsection{Syzygies and  quartics singular along paracanonical curves}
\label{subsecquartic}
We first discuss an alternative characterization of the non-surjectivity of the map $\mu_{C,L}$:

\begin{prop}\label{proquarticsing} Assume the paracanonical curve $\phi_L(C)$ is projectively normal
 and scheme-theoretically cut out by quadrics. Then $K_{1,2}(C,L)\neq0$ if and only if there exists a degree $4$ homogeneous polynomial
on $\PP^6$, which vanishes to order at least $2$ along $C$ but does not belong
to the image of the multiplication map
${\rm Sym}^2 I_{C,L}(2)\rightarrow I_{C,L}(4)$.
\end{prop}

\begin{proof}
We work on the variety $X\stackrel{\tau}{\rightarrow} \PP^6$ defined as the blow-up of $\PP^6$ along $\phi_L(C)$.
Let $E$ be the exceptional divisor of the blow-up, and consider the line bundle
$H:=\tau^*\mathcal{O}_{\PP^6}(2)(-E)$ on $X$. Its space of sections identifies to $I_C(2)$, and our assumption that $C$ is scheme-theoretically cut out by quadrics says equivalently that
$H$ is a  globally generated  line bundle on $X$. The nonvanishing of $K_{1,2}(C,L)$ is equivalent to the non-surjectivity
of the
multiplication map
\begin{eqnarray}\label{eqmult25jan1}
H^0(X,H)\otimes H^0(X,\tau^*\mathcal{O}(1))\rightarrow H^0(X,H\otimes \tau^*\mathcal{O}(1)),
\end{eqnarray}
where we use the identification
$$H^0\bigl(X,H\otimes \tau^*\mathcal{O}(1)\bigr)=H^0\bigl(X,\tau^*\mathcal{O}(3)(-E)\bigr)=I_C(3).$$
As $H$ is globally generated by its space $W:=I_C(2)$ of global sections,
 the Koszul complex
\begin{eqnarray}\label{eqkoscom}
0\rightarrow\bigwedge^7 W\otimes\mathcal{O}_X(-7H)\rightarrow \ldots\rightarrow
\bigwedge^2W \otimes\mathcal{O}_X(-2H)\rightarrow W\otimes \mathcal{O}_X(-H)\rightarrow
  \mathcal{O}_X\rightarrow0
\end{eqnarray}
is exact.
We now twist this complex by $\tau^*\mathcal{O}_{\PP^6}(1)(H)$ and take global sections. The last
map is then the multiplication map (\ref{eqmult25jan1}).
The successive terms of this twisted complex are
$$\bigwedge^ iW\otimes\mathcal{O}_X\bigl(\tau^*\mathcal{O}(1)\bigr)((-i+1)H),$$
for $0\leq i\leq 7$.
The spectral sequence abutting
to the hypercohomology of this complex, that is $0$, has
\begin{eqnarray}
\label{E200} E_2^{0,0}={\rm coker}\,\Bigl\{W\otimes H^0(X,\tau^*\mathcal{O}(1))\rightarrow H^0(X,H\otimes \tau^*\mathcal{O}(1))\Bigr\}\end{eqnarray}
and the terms
$E_1^{i,-i-1}$ for $i< -1$ are equal to
$\bigwedge^{-i}W\otimes H^{-i-1}\bigl(X,\tau^*\mathcal{O}(1)((i+1)H)\bigr)$.
Similarly, we have
$$E_1^{i,-i}=\bigwedge^{-i}W\otimes H^{-i}\bigl(X,\tau^*\mathcal{O}(1)((i+1)H)\bigr).$$

\begin{lemm} \label{levan}
\begin{itemize}
\item[\rm (i)]  We have
\begin{eqnarray}\label{eqcohgroup} E_1^{i,-i-1}=\bigwedge^{-i}W\otimes H^{-i-1}\bigl(X,\tau^*\mathcal{O}(1)((i+1)H)\bigr)=0,
\end{eqnarray}
for $-i-1=5,\ldots, 1$.
\item[\rm (ii)]
For $-i-1=6$, that is, $i=-7$, we have
\begin{eqnarray}\label{eqcohgroup2}E_1^{-7,6}=\bigwedge^7W\otimes H^{6}\bigl(X,\tau^*\mathcal{O}(1)(-6H)\bigr)=
\bigwedge^7W\otimes I_C(4)_2^{\vee}
,\end{eqnarray}
 where $I_C(4)_2\subseteq I_C(4)$
is the set of quartic polynomials vanishing at order at least $2$ along $C$,
and
\begin{eqnarray}\label{eqcohgroup3}E_1^{-6,6}=\bigwedge^6W\otimes H^{6}\bigl(X,\tau^*\mathcal{O}(1)(-5H)\bigr)=
\bigwedge^6W\otimes I_C(2)^{\vee}.
\end{eqnarray}
\item[\rm (iii)] We have $E_1^{i,-i}=0$, for $-6<i<0$.
\end{itemize}
\end{lemm}
\begin{proof}[Proof of Lemma \ref{levan}] (i) We want equivalently to show that
$$H^{\ell}(X,\tau^*\mathcal{O}(1)(-\ell H))=0,\,\,\mbox{ when } \ell=5,\ldots, 1.$$
Recall that
$H=\tau^*\mathcal{O}(2)(-E)$. Furthermore,
\begin{eqnarray}\label{eqK}
K_X=\tau^*\mathcal{O}_{\PP^6}(-7)(4E).
\end{eqnarray}
So we have to prove that
\begin{eqnarray}\label{eqcohgroup16feb}H^{\ell}\bigl(X,\tau^*\mathcal{O}(-2\ell+1)(\ell E)\bigr)=0,\,\mbox{ for }  \ell=5,\ldots, 1.
\end{eqnarray}
Examining the spectral sequence induced by  $\tau$, and using the fact that
$$R^s\tau_*(\mathcal{O}_X(tE))=0$$
for $s\not=0,\,4$ and also for $s=4,\,t\leq4$,
we see that for $1\leq \ell\leq 4$,
$$H^{\ell}\bigl(X,\tau^*\mathcal{O}(-2\ell+1)(\ell E)\bigr)=H^\ell\bigl(\PP^6,\mathcal{O}(-2\ell+1)\otimes R^0\tau_*\mathcal{O}_X(\ell E)\bigr).$$
For $1\leq \ell\leq 4$, the right hand side is zero, because it is equal to $H^\ell\bigl(\PP^6, \mathcal{O}(-2\ell+1)\bigr)$.

\vskip 4pt

For $\ell=5$, we have  to compute the space
$H^5(X,\tau^*\mathcal{O}(-9)(5E))$, which by Serre duality and by (\ref{eqK}),
is dual to the  space
$$H^1(X,\tau^*\mathcal{O}(2)(-E))=H^1(\PP^6,\mathcal{O}(2)\otimes\mathcal{I}_C)=0.$$

(ii) We have to compute the spaces
$H^{6}(X,\tau^*\mathcal{O}(1)(-6H))$ and $H^{6}(X,\tau^*\mathcal{O}(1)(-5H))$.
As $H:=\tau^*\mathcal{O}(2)(-E)$, this is rewritten as
$H^{6}(X,\tau^*\mathcal{O}(-11)(6E))$ and $H^{6}(X,\tau^*\mathcal{O}(-9)(5E))$ respectively.
If we dualize using (\ref{eqK}), we get
$$H^{6}\bigl(X,\tau^*\mathcal{O}(-11)(6E)\bigr)^{\vee}=H^0\bigl(X,\tau^*\mathcal{O}(4)(-2E)\bigr)=I_C(4)_2,$$
$$H^{6}\bigl(X,\tau^*\mathcal{O}(-9)(5E)\bigr)^{\vee}=H^0\bigl(X,\tau^*\mathcal{O}(2)(-E)\bigr)=I_C(2).$$
(iii) We have
$$E_1^{i,-i}=E_1^{-6,6}=\bigwedge^{-i}W\otimes H^{-i}\bigl(X,\tau^*\mathcal{O}(1)((i+1)H)\bigr)=
\bigwedge^{-i}W\otimes H^{-i}\bigl(X,\tau^*\mathcal{O}(2i+3)((-i-1)E)\bigr).$$
For $1\leq - i\leq 5$, we have $R^s\tau_*\mathcal{O}_X\bigl((-i-1)E\bigr)=0$ unless $s=0$. Furthermore,
we have $$R^0\tau_*\mathcal{O}_X\bigl((-i-1)E\bigr)=\mathcal{O}_{\PP^6},$$ so that
$$H^{-i}\bigl(X,\tau^*\mathcal{O}(2i+3)((-i-1)E)\bigr)=H^{-i}\bigl(\PP^6,\mathcal{O}_{\PP^6}(2i+3)\bigr)=0.$$
\hfill $\Box$
\end{proof}

\begin{coro}\label{corovane2} Only one  $E_2^{p,q}$-terms of this spectral sequence is possibly nonzero
in degree $-1$, namely
\begin{eqnarray}\label{E276}E_2^{-7,6}={\rm Ker}\Bigl\{\bigwedge^7W\otimes  I_C(4)_2^{\vee}\rightarrow
\bigwedge^6W\otimes I_C(2)^{\vee}\Bigr\}.
\end{eqnarray}
Furthermore,
all the differentials $d_r$ starting from $E_2^{-7,6}$ vanish for $2\leq r<7$.
\end{coro}

Note that the map
$$\bigwedge^7W\otimes  I_C(4)_2^{\vee}\rightarrow
\bigwedge^6W\otimes I_C(2)^{\vee}$$ is nothing but the transpose of the multiplication map
$$W\otimes I_C(2)\rightarrow I_C(4)_2,$$
up to trivialization of $\bigwedge^7W$. It follows that
\begin{eqnarray}\label{E276dual}(E_2^{-7,6})^{\vee}={\rm Coker}\Bigl\{W\otimes
I_C(2)  \rightarrow I_C(4)_2\Bigr\}.
\end{eqnarray}
Corollary \ref{corovane2} concludes the  proof of the proposition since it implies that
 we have an isomorphism given by
$d_7$ between (\ref{E276}) and (\ref{E200}), or a perfect duality between
(\ref{E276}) and the cokernel (\ref{E276dual}).
\hfill $\Box$
\end{proof}

Proposition \ref{proquarticsing}
has the following consequence. Recall that $P_8^{14}$ is the moduli space of pairs
$[C,L]$, with $C$ being a smooth curve of genus $8$ and $L\neq K_C$ a paracanonical line bundle.

\begin{prop}\label{splitting} The Koszul divisor $\mathfrak{Kosz}$ of $P_8^{14}$  is the union of two divisors, one of them being the set of pairs
$[C,L]$ such that $\phi_L(C)$ is not scheme-theoretically cut out by
quadrics, the other being the set of pairs $[C,L]$ such that $H^1(\PP^6,\mathcal{I}_C^2(4))\not=0$, or equivalently, such that
there exists a quartic  which is singular along $\phi_L(C)$ but does not lie in
${\rm Sym}^2 I_C(2)$.
\end{prop}
\begin{proof} We first have to prove that the locus of pairs $[C,L]$ such that
$\phi_L(C)$ is not scheme-theoretically cut-out by quadrics is contained in the divisor
$\mathfrak{Kosz}$. This is a consequence of the following lemmas:

\begin{lemm} \label{legencubic} If $L\not=K_C$ is a projectively normal paracanonical line bundle
  on a curve of genus $8$, then $\phi_L(C)$ is scheme-theoretically cut out by cubics.
\end{lemm}
\begin{proof}[Proof of Lemma \ref{legencubic}] We observe that the twisted ideal sheaf $\mathcal{I}_C(3)$ is regular in Castelnuovo-Mumford sense.
Indeed, we have $$H^i(\PP^6,\mathcal{I}_C(3-i))=H^{i-1}(C,L^{\otimes (3-i)})$$ for $i\geq 2$, and
the right hand side is obviously  $0$ for $i-1\geq 2$, and also $0$ for $i-1=1$ since
$H^1(C,L)=0$ because $ L\not=K_C$ and ${\rm deg}\,L=2g-2$.
 For $i=1$, we have
$$H^1(\PP^6,\mathcal{I}_C(2))=0$$
by projective normality.
Being regular, the sheaf $\mathcal{I}_C(3)$ is generated by global sections.
\hfill $\Box$
\end{proof}
\begin{coro}\label{cubics1} If $C,\,L$ are as above, and $C$ is not scheme-theoretically cut out by quadrics, then the multiplication map
$$I_C(2)\otimes H^0(\PP^6, \mathcal{O}_{\PP^6}(1))\rightarrow I_C(3)$$
is not surjective.
\end{coro}

To conclude the proof of the proposition, we just have to show that the sublocus of $P_8^{14}$
where $L$ is not projectively normal is not a divisor, since the statement of the proposition will be then an immediate consequence of Proposition \ref{proquarticsing}. We argue along the lines of \cite{GL}.
First of all, a line bundle $L$ of degree $14$ is not generated by sections if and only if
$L=K_C(-x+y)$ for some points $x,\,y\in C$. This determines a codimension
$6$ locus of $P_{8}^{14}$. Similarly
$L$ is not very ample if and only if $L=K_C(-x-y+z+t)$, for some points $x,\,y,\,z,\,t$ of $C$, which is satisfied in  a codimension $4$ locus of $P_{8}^{14}$.
Finally, assume $L$ is very ample but $\phi_L(C)$ is not projectively normal. Equivalently
$${\rm Sym}^2H^0(C,L)\rightarrow H^0(C,L^{\otimes 2})$$
is not surjective, which means that there exists
a rank $2$ vector bundle $F$ on $C$ which is a nontrivial extension
$$0\longrightarrow K_C\otimes L^{\vee}\longrightarrow F \longrightarrow L\longrightarrow 0,$$
such that
$h^0(C,F)=7$.
If $x,\,y,\,z\in C$, there exists a nonzero section $\sigma\in H^0(C, F)$ vanishing
on $x,\,y$ and $z$,
and thus $F$ is also an extension
\begin{eqnarray} \label{eqext8ev} 0\longrightarrow  D\longrightarrow F\longrightarrow K_C\otimes D^{\vee}\longrightarrow 0,
\end{eqnarray}
where $D$ is  a line bundle such that
$h^0(C,D(-x-y-z))\not=0$, and $h^0(C,L\otimes D^{\vee})\not=0$.
We thus have $h^0(C,D)+h^0(C,K_C\otimes D^{\vee})\geq 7$ and ${\rm Cliff}(D) \leq 2$. As $D$ is effective of degree
at least $3$, one has the following possibilities:

\vskip 4pt

a) $h^0(C,K_C\otimes D^{\vee})=0$, and then $D=L$, which contradicts the fact that the extension (\ref{eqext8ev})
is not split;

b)  $h^0(C,K_C\otimes D^{\vee})=1$ and $h^0(C,D)\geq 6$, and then
$D=L(-x)$ and $h^0(K_C\otimes L^{\vee}(x))\not=0$, so $L=K_C(x-y)$, which happens in a locus of codimension at least $6$
in $P_8^{14}$;

c) $D$ contributes to the Clifford index of $C$. As the locus of curves $[C]\in \cM_8$ with $\mbox{Cliff}(C) \leq 2$ is of codimension $2$ in $\mathcal{M}_8$, this situation does not occur in codimension $1$.
\hfill $\Box$
\end{proof}

\vskip 4pt

We shall need later on the following result:

\begin{lemm}\label{leinjmult} Let $\phi_L : C \hookrightarrow\PP^6$ be a projectively normal paracanonical curve of genus $8$. If $C$
is scheme-theoretically cut out by quadrics, the multiplication map
\begin{eqnarray}
\label{eqmult6fev} {\rm Sym}^2 \ I_{C,L}(2) \rightarrow  I_{C,L}(4)
\end{eqnarray}
is injective.
\end{lemm}

\begin{proof} As the restriction map
$\phi_L^*: H^0(\PP^6, \mathcal{O}_{\PP^6}(2)) \rightarrow  H^0(C, L^{\otimes 2})$
is surjective, its kernel $I_{C,L}(2)$ is of dimension $7$. Let as before
$\tau: X \rightarrow \PP^6$
be the blow-up of
$\PP^6$ along $\phi_L(C)$, and let $E$ be its exceptional divisor. We view $I_{C,L}(2)$
 as $H^0(X,\tau^*\mathcal{O}(2)(-E))$
and our assumption is that $I_{C,L}(2)$  generates the line bundle
$H:=\tau^*\mathcal{O}(2)(-E)$ everywhere on
$X$. Thus $I_{C,L}(2)$ provides a morphism
  \begin{eqnarray}
  \label{sec0}\psi: X \rightarrow \PP(I_{C,L}(2)).
  \end{eqnarray}
Now we have ${\rm deg}\, c_1(H)^6 \not=0$
 by Sublemma \ref{lenumbers} below, and thus the morphism $\psi$
has to be generically finite, hence dominant since both spaces have dimension $6$. It is then
clear that the pull-back map

$$\psi^*: H^0\bigl(\PP(I_{C,L}(2)),\mathcal{O}(2)\bigr) \rightarrow H^0(X, H^{\otimes 2})$$
is injective. On the other hand, this morphism is nothing but the map (\ref{eqmult6fev}).
\hfill $\Box$
\end{proof}

\begin{sublemm} \label{lenumbers} With the same notation as above, we have
\begin{eqnarray}
\label{eqdeg6fev}{\rm deg}\, c_1(H)^6=8.
\end{eqnarray}
\end{sublemm}

\begin{proof}
We have
$$c_1(H
)^6 =\sum_i {6\choose i}(-2)^ih^i \cdot E^{6-i},$$
where $h :=\tau^*c_1(\mathcal{O}_{\PP^6}(1))$, and
$$h^i\cdot  E^{6-i} = 0$$
for $i \not= 6,\, 1,\,0$. Furthermore
$$h^6 = 1,\mbox{ and }  h\cdot   E^5 = {\rm deg}\,\phi_L(C) = 14$$
and $E^6 = c_1(N_C)$. By adjunction formula
$${\rm deg}\,c_1(N_C) = 7 {\rm deg}\, \phi_L(C) + {\rm deg}\,K_C = 8 \cdot 14.$$
It follows that
$${\rm deg}\, c_1(H)^6
= 64-6\cdot 28+8 \cdot 14=8,
$$
which proves (\ref{eqdeg6fev}).
\hfill $\Box$
\end{proof}

Proposition \ref{proquarticsing} and Lemma \ref{leinjmult} describe precisely the splitting of the Koszul divisor $\mathfrak{Kosz}$ into the divisors $\mathfrak{Kosz}_6$
and $\mathfrak{Kosz}_7$ corresponding to paracanonical curves $[C,L]\in P^{14}_8$ having a non-zero syzygy $\gamma\in K_{1,2}(C,L)$ of rank $6$ or respectively $7$.
Precisely, $\mathfrak{Kosz}_6$ is a unirational divisor (cf. \cite{chiodetal} Theorem 5.3) consisting of those paracanonical curves $C\subseteq \PP^6$ for which
$H^1(\PP^6, \mathcal{I}_C^2(4))\neq 0$. The divisor $\mathfrak{Kosz}_7$ consists of paracanonical curves $C\subseteq \PP^6$ which are not scheme-theoretically cut out by quadrics.

\vskip 4pt

\section{First proof: reducible spin curves}\label{secthirdproof}
\subsection{The syzygy is degenerate}

The first observation is the following result (already observed experimentally in \cite{chiodetal}), which turns out to be useful for the
description given below of the general paracanonical curve of genus $8$ with nontrivial syzygies.

\begin{prop} \label{prodeg} Let $C\subseteq  \PP^6$ be a smooth paracanonical curve  of genus $8$ and degree $14$, scheme-theoretically generated by quadrics.
Then a nontrivial syzygy $$\gamma \in {\rm Ker}\, \bigl\{I_C(2)\otimes H^0(\mathcal{O}_{\PP^6}(1))\rightarrow I_C(3)\bigr\}$$
must be degenerate, that is of rank at most $6$.
\end{prop}

\begin{proof}  We use the morphism
$$\psi:X\rightarrow \PP(I_C(2))$$
introduced in  (\ref{sec0}), where
$\tau:X\rightarrow \PP^6$ is the blow-up of $C$ with exceptional divisor $E$,
and $H:=\tau^*\mathcal{O}_{\PP^6}(-2E)$.
This gives us a morphism
$$(\tau,\psi):X\rightarrow \PP^6\times \PP^6$$
which is of degree $1$ on its image,
and the syzygy $\gamma$ induces a hypersurface
$Y$  of bidegree $(1,1)$ in $\PP^6\times \PP^6$ containing
the $6$-dimensional variety $(\tau,\psi)(X)$.
Assume to the contrary that $\gamma$ has maximal rank $7$, or equivalently that $Y$ is smooth.
Then by the Lefschetz Hyperplane Restriction Theorem, the restriction map
$H^{10}(\PP^6\times \PP^6,\mathbb{Z})\rightarrow H^{10}(Y,\mathbb{Z})$ is surjective, so that
$[(\tau,\psi)(X)]_Y\in H^{10}(Y,\mathbb{Z})$ is the restriction of  a class $\beta\in H^{10}(\PP^6\times \PP^6,\mathbb{Z})$, which implies that
\begin{eqnarray} [(\tau,\psi)(X)]=\beta\cdot [Y]\,\,{\rm in}\,\,H^{12}(\PP^6\times \PP^6,\mathbb{Z}),
\end{eqnarray}
where $[Y]\in H^2(\PP^6\times \PP^6,\mathbb{Z})$ is the class of $Y$, that is
$h_1+h_2$, with $h_i$ for  $i=1,2$ being the pull-backs of the hyperplane classes on each factor.
Note that $H^{12}(\PP^6\times \PP^6,\mathbb{Z})$ is the set of degree
$6$ homogeneous polynomials with integral coefficients in $h_1$ and $h_2$.
We now have:
\begin{lemm}\label{ledivclass} An element  $\alpha\in H^{12}(\PP^6\times \PP^6,\mathbb{Z})$ is of the form
$(h_1+h_2)\cdot \beta$ if and only if it satisfies
the condition
\begin{eqnarray}\label{degcond23sep}\sum_{i=0}^6 (-1)^i h_1^{i}\cdot h_2^{6-i}\cdot\alpha=0\,\,{\rm in}\,\,H^{24}(\PP^6\times \PP^6,\mathbb{Z})=\mathbb{Z}.
\end{eqnarray}
\end{lemm}
\begin{proof}[Proof of Lemma \ref{ledivclass}] We have $(h_1+h_2)\cdot \bigl(\sum_i(-1)^ih_1^{i}\cdot h_2^{6-i}\bigr)=0$ in $H^{14}(\PP^6\times \PP^6,\mathbb{Z})$, so one implication is obvious. That the two conditions are equivalent then follows
from the fact that both conditions  determine a saturated corank $1$ sublattice of $H^{12}(\PP^6\times \PP^6,\mathbb{Z})$.
\hfill $\Box$
\end{proof}

To conclude that $\gamma$ has to be degenerate,  in view of Lemma \ref{ledivclass},
it suffices to prove that the class $[(\tau,\psi)(X)]$ does not satisfy (\ref{degcond23sep}). Since
$(\tau,\psi)^*h_1=c_1(H)$ and $(\tau,\psi)^*h_2=2c_1(H)-E$, it is enough to prove that
$$\sum_{i=0}^6 (-1)^ic_1(H)^i\cdot (2c_1(H)-E)^{6-i}\not=0,$$
which follows from the computations made in the proof of Sublemma \ref{lenumbers}.
\hfill $\Box$
\end{proof}

\subsection{Syzygies and spin curves of genus 22 in $\PP^6$}

Recall that $\overline{\mathcal{S}}_g^{-}$ denotes the moduli stack of odd stable spin curves of genus $g$,   see \cite{Cor} for details.
We start with a nodal genus $22$ spin curve of the form $[D:=C \cup E, \vartheta]\in \overline{\mathcal{S}}_{22}^{-}$, where
 $C$ is a smooth genus $8$ curve, $E$ is a smooth elliptic curve and  $e:=C\cap E$ consists of $14$ distinct points, thus $p_a(D)=22$.  Assume  $\vartheta\in \mbox{Pic}^{21}(D)$ verifies $\vartheta^{\otimes 2}\cong \omega_D$, hence the restricted line bundles $\vartheta_{E}$ and $\vartheta_{C}$  have degrees $7$ and $14$ respectively.  Furthermore, $h^0(E, \vartheta_E)=7$, whereas  $h^0(C, \vartheta_C)=7$ if and only if $\vartheta_C\ncong K_C$. The intersection divisor $e$ on the two components of $D$ is characterized by
 $$e_C\in |\vartheta_C^{\otimes 2} \otimes K_C^{\vee}| \ \mbox{ and } \ e_E\in |\vartheta_E^{\otimes 2}|.$$
Note in particular that $e_C\in|K_C|$ if and only if $\vartheta_{C}^{\otimes 2}=K_C^{\otimes 2}$, that is
$(C,\vartheta_{C})$ is canonical or Prym canonical.

The line bundle $\vartheta$ on $D$ fits into the Mayer-Vietoris exact sequence:
 $$
 0 \longrightarrow  \vartheta \longrightarrow \vartheta_C \oplus \vartheta_E \stackrel{r}  \longrightarrow  \mathcal O_e(\vartheta) \longrightarrow 0,
$$
where $r$ is defined by the isomorphisms on the fibers of $\vartheta_C$ and $\vartheta_E$ over the points in $e$. Given $\vartheta_C\in \mbox{Pic}^{14}(C)$  with $\vartheta_C^{\otimes 2}=K_C(e)$ and $\vartheta_E \in \mbox{Pic}^7(E)$ with $\vartheta_E^{\otimes 2}=\mathcal{O}_E(e)$, there is a  finite number of stable spin curves  $[D,\theta]\in \overline{\mathcal{S}}_{22}^{-}$ such that the restrictions of $\vartheta$ to $C$ and $E$ are isomorphic to $\vartheta_C$ and $\vartheta_E$ respectively.  Passing to  global sections in the Mayer-Vietoris sequence, we obtain the exact sequence:
\begin{eqnarray}\label{exact}
0 \longrightarrow  H^0(D, \vartheta) \longrightarrow H^0(C, \vartheta_C) \oplus H^0(E, \vartheta_E )\stackrel{r}  \longrightarrow H^0(\mathcal O_e(\vartheta)) \longrightarrow \cdots.
\end{eqnarray}
Note that $r$ is represented by a  $14 \times 14$ matrix and $ h^0(D, \vartheta) = 14-\mbox{rk}(r) $. In the case of a reducible spin curve coming from the syzygy of a paracanonical genus 8 curve in $\mathfrak{Kosz_6}$, one has $h^0(D, \vartheta) =\mbox{rk}(r) =7$.

\subsection{Proof of Theorem \ref{theomainintro} via reducible spin curves}

Theorem \ref{theomainintro} states that every Prym canonical curve of genus 8 has a syzygy of rank 6. First we observe the existence  of such a curve having the generic behavior described in  Lemma \ref{DCE}.

\begin{lemm}\label{start}
There exists a curve $[C,\eta]\in \cR_8$, whose Prym canonical model is  scheme theoretically cut out by  quadrics,  and  $K_{2,1}(C, K_C\otimes \eta)$ is $1$-dimensional, generated by a syzygy $\gamma$ of rank $6$. The syzygy scheme of $\gamma$ is the union of a point $p$ and a nodal curve $D=C\cup E$,  such that $E$ is a smooth elliptic curve of degree $7$ and $e:=C\cdot E\in |K_C|$ consists of  $14$ mutually distinct points.
 Moreover, no cubic polynomial on $\PP^6$ vanishes with multiplicity $2$ along $C$.
\end{lemm}

\begin{proof} Examples of \emph{singular} Prym canonical curves  having all these properties have been produced in  \cite{chiodetal} Proposition 4.4 or \cite{cofre}. A generic deformation in $\rr_8$ of these singular examples will provide the required smooth Prym canonical curve.
\hfill $\Box$
\end{proof}

\begin{proof}[(First) proof of Theorem \ref{theomainintro}]

We denote by $X$ the moduli space of elements $[C,\eta, x_1, \ldots, x_{14}]$, where $[C,\eta]\in \cR_8$ is a smooth Prym curve of genus $8$  and $x_i\in C$
are pairwise distinct points such that  $x_1+\cdots+x_{14}\in |K_C|\cong \PP^7$. Since the fibres of the forgetful map $X\rightarrow \cR_8$ are $7$-dimensional, it
follows that $X$ is an irreducible variety of dimension $28$.

\vskip 3pt
Let $T$ be the locally closed parameter space of  odd genus $22$ spin curves  having  the form
$$\Bigl(\bigl[D:=C\cup_{\{x_1, \ldots, x_{14}\}} E, \vartheta \bigr]:[C]\in \cM_8, \ \sum_{i=1}^{14} x_i\in |K_C|, \ [E, x_1, \ldots, x_{14}]\in \cM_{1,14}, \  \vartheta^{\otimes 2}=\omega_D\Bigr).$$

Observe that points in $T$, apart from the spin structure $[D,\vartheta]\in \overline{\mathcal{S}}_{22}^{-}$ also carry an underlying Prym structure $[C,\eta:=K_C\otimes \vartheta_C^{\vee}]\in \cR_8$, for $\vartheta_C^{\otimes 2}\cong K_C(x_1+\cdots+x_{14})\cong K_C^{\otimes 2}$. One has an induced  finite morphism  $T\rightarrow X\times \cM_{1,14}$, as well as a map $\mu:T\rightarrow \cR_8$ forgetting the $14$-pointed elliptic curve. It follows  that $\mbox{dim}\,T=\mbox{dim}\,X+\mbox{dim}\,\cM_{1,14}=42$. The locus
$$T_{7}:=\bigl\{[D, \vartheta]\in T: h^0(D,\vartheta)\geq 7\bigr\}$$
has the structure of a skew-symmetric degeneracy locus. Applying \cite{Harris} Theorem 1.10, each component of $T_7$ has codimension at most ${7\choose 2}=21$ inside $T$, that is, $\mbox{dim}(T_7)\geq \mbox{dim}(\cR_8)$.

\vskip 5pt

By passing to a general $8$-nodal Prym canonical curve $[C,\eta]$, following \cite{chiodetal} Proposition 4.5, as well as Lemma \ref{start}, we have that $\mbox{dim } K_{1,2}(C,K_C\otimes \eta)=1$. In particular, the fibre $\mu^{-1}([C,\eta])$ contains an isolated point, which shows that $T_7$ is non-empty and has a component which maps dominantly under $\mu$ onto $\cR_8$.    Theorem \ref{theomainintro} now follows.
\hfill $\Box$
\end{proof}

\begin{rema}{\rm
The same construction can be carried out at the level of general paracanonical curves $[C,L]\in P^{14}_8$, where $L\in \mbox{Pic}^{14}(C)\setminus \{K_C\}$.   The key difference is that we replace $T$ by a variety $T'$ parametrizing objects
$$\Bigl(\bigl[D:=C\cup_{\{x_1, \ldots, x_{14}\}} E, \vartheta, L \bigr]: [C,x_1, \ldots, x_{14}]\in \cM_{14,8}, \  L\in \mbox{Pic}^{14}(C)\setminus \{K_C\}, $$
$$\ \sum_{i=1}^{14} x_i\in |L^{\otimes 2}\otimes K_C^{\vee}|, \   [E, x_1, \ldots, x_{14}]\in \cM_{1,14}, \vartheta^{\otimes 2}=\omega_D\Bigr).$$
Similarly, we have a morphism $\mu':T'\rightarrow P^{14}_8$ retaining the pair $[C,L]$ alone. The main difference compared to the Prym canonical case is that now
$$\mbox{dim } |L^{\otimes 2}\otimes K_C^{\vee}|=6,$$
therefore $\mbox{dim}(T')=\mbox{dim}(P^{14}_8)+\mbox{dim}(\cM_{1,14})+6=49$.
The degeneracy locus $T_7'\subseteq T'$ defined by the condition $h^0(D,\vartheta)\geq 7)$ has codimension $21$ inside $T'$, that is, $$\mbox{dim}(T_7')=28=\mbox{dim}(P_8^{14})-1.$$ It follows that the image $\mu'(T_7')\subseteq P^{14}_8$ has codimension $1$, which is in accordance with
$\mathfrak{Kosz_6}$ being a divisor in  $P^{14}_{8}$.
}
\end{rema}
\section{Second proof: Divisor class calculations on $\rr_g$}\label{secfirstproof}

Recall \cite{FL} that $\rr_g$ is the Deligne-Mumford moduli space of Prym curves of genus $g$, whose geometric points are triples $[X, \eta, \beta]$, where $X$ is a quasi-stable curve of genus $g$, $\eta\in \mbox{Pic}(X)$ is a line bundle of total degree  $0$ such that $\eta_{E}=\OO_E(1)$ for each smooth rational component $E\subseteq X$ with $|E\cap \overline{X-E}|=2$  (such a component is said to be \emph{exceptional}), and $\beta:\eta^{\otimes 2}\rightarrow \OO_X$ is a sheaf homomorphism whose restriction to any non-exceptional component is an isomorphism. If $\pi:\rr_g\rightarrow \mm_g$ is the map dropping the Prym structure, one has the formula
\begin{equation}\label{pullbackrg}
\pi^*(\delta_0)=\delta_0^{'}+\delta_0^{''}+2\delta_{0}^{\mathrm{ram}}\in CH^1(\rr_g),
\end{equation}
where $\delta_0^{'}:=[\Delta_0^{'}], \, \delta_0^{''}:=[\Delta_0^{''}]$, and $\delta_0^{\mathrm{ram}}:=[\Delta_0^{\mathrm{ram}}]$ are irreducible boundary divisor classes on $\rr_g$, which we describe by specifying their respective general points.

\vskip 4pt

We choose a general point $[C_{xy}]\in \Delta_0\subset \mm_g$ corresponding to a smooth $2$-pointed curve $(C, x, y)$ of genus $g-1$ and consider the normalization map $\nu:C\rightarrow C_{xy}$, where $\nu(x)=\nu(y)$. A general point of $\Delta_0^{'}$ (respectively of $\Delta_0^{''}$) corresponds to a pair $[C_{xy}, \eta]$, where $\eta\in \mbox{Pic}^0(C_{xy})[2]$ and $\nu^*(\eta)\in \mbox{Pic}^0(C)$ is non-trivial
(respectively, $\nu^*(\eta)=\OO_C$). A general point of $\Delta_{0}^{\mathrm{ram}}$ is a Prym curve of the form $(X, \eta)$, where $X:=C\cup_{\{x, y\}} \PP^1$ is a quasi-stable curve with $p_a(X)=g$ and  $\eta\in \mbox{Pic}^0(X)$ is a line bundle such that $\eta_{\PP^1}=\OO_{\PP^1}(1)$ and $\eta_C^{2}=\OO_C(-x-y)$. In this case, the choice of the homomorphism $\beta$ is uniquely determined by $X$ and $\eta$. In what follows, we work on the partial compactification $\pr_g\subseteq \rr_g$ of $\cR_g$ obtained by removing the boundary components
$\pi^{-1}(\Delta_j)$ for $j=1, \ldots, \lfloor \frac{g}{2}\rfloor$, as well as $\Delta_0^{''}$. In particular, $CH^1(\pr_g)=\mathbb Q \langle \lambda, \delta_0^{'}, \delta_0^{\mathrm{ram}}\rangle$.

\vskip 4pt

For a stable Prym curve $[X, \eta]\in \pr_g$, set $L:=\omega_X\otimes \eta\in \mbox{Pic}^{2g-2}(X)$ to be the paracanonical bundle. For $i\geq 1$, we introduce the vector bundle $\cN_k$  over $\pr_g$, having fibres
$$\cN_k[X, \eta]=H^0(X, L^{\otimes k}).$$
The first Chern class of $\cN_k$ is computed in \cite{FL} Proposition 1.7:
\begin{equation}\label{Ni}
c_1(\cN_k)={k\choose 2}\Bigl(12\lambda-\delta_0^{'}-2\delta_0^{\mathrm{ram}}\Bigr)+\lambda-\frac{k^2}{4} \delta_0^{\mathrm{ram}}.
\end{equation}
Then we define the locally free sheaves $\mathcal{G}_k$ on $\pr_g$ via the exact sequences
$$0\longrightarrow \G_k\longrightarrow \mbox{Sym}^k \cN_1\longrightarrow \cN_k \longrightarrow 0, $$
that is, satisfying $\G_k[X, \eta]:=I_{X,L}(k)\subseteq \mbox{Sym}^k H^0(X,L)$. Using (\ref{Ni}) one  computes $c_1(\G_k)$.

\vskip 4pt

We also need the class of the vector bundle $\G$ with fibres $$\G[X, \eta]=H^0(X, \omega_X^{\otimes 5}\otimes \eta^{\otimes 4})=H^0(X, \omega_X\otimes L^{\otimes 4}).$$

\begin{lemm}
One has $c_1(\G)=121\lambda-10\delta_0^{'}-24\delta_0^{\mathrm{ram}}\in CH^1(\pr_g)$.
\end{lemm}
\begin{proof}
We apply Grothendieck-Riemann-Roch to the universal Prym curve $f:\cC\rightarrow \pr_g$. Denote by $\L\in \mbox{Pic}(\cC)$ the universal \emph{Prym bundle}, whose restriction to each
Prym curve is the corresponding $2$-torsion point, that is, $\L_{|f^{-1}([X, \eta])}=\eta$, for each point $[X, \eta]\in \pr_g$. Since $R^1 f_*(\omega_f^{\otimes 5} \otimes \L^{\otimes 4})=0$, we write
$$c_1(\G)=f_*\Bigl[\Bigl(1+5c_1(\omega_f)+4c_1(\L)+\frac{(5c_1(\omega_f)+4c_1(\L))^2}{2}\Bigr)\cdot \Bigl(1-\frac{c_1(\omega_f)}{2}+\frac{c_1^2(\Omega_f^1)+[\mathrm{Sing}(f)]}{12}\Bigr)\Bigr]_2.$$
We use then the formulas $f_*(c_1^2(\L))=-\delta_0^{\mathrm{ram}}/2$ and $f_*(c_1(\L)\cdot c_1(\omega_f))=0$ (see \cite{FL}, Proposition 1.6) coupled with Mumford's formula $f_*(c_1^2(\Omega_f^1)+[\mbox{Sing}(f)])=12 \lambda$ as well with the identity
$$\kappa_1:=f_*(c_1^2(\omega_f))=12\lambda-\delta_0^{'}-2\delta_0^{\mathrm{ram}},$$ in order to conclude.
\hfill $\Box$
\end{proof}

The Koszul locus
$$\cZ_8:=\mathfrak{Kosz}\cap \cR_8=\Bigl\{[C, \eta]\in \cR_8: K_{1,2}(C,K_C\otimes \eta)\neq 0\Bigr\}$$ is a virtual divisor on $\cR_8$, that is, the degeneracy locus of a map between vector bundles of the same rank over $\pr_8$. If it is a genuine divisor (which we aim to rule out), the class of its closure in $\pr_8$ is given by \cite{chiodetal} Theorem F:
$$[\overline{\cZ}_8]=27\lambda-4\delta_0^{'}-6\delta_0^{\mathrm{ram}}\in CH^1(\pr_8).$$

\begin{rema}
{\rm Some of the considerations above can be extended to higher order torsion points. We recall that $\cR_{g,\ell}$ is the moduli space of pairs $[C,\eta]$, where $C$ is a smooth curve of genus $g$ and $\eta\in \mbox{Pic}^0(C)$ is a non-trivial $\ell$-torsion point. It is then shown in \cite{chiodetal} that the locus $\cZ_{8,\ell}:=\mathfrak{Kosz}\cap \cR_{8,\ell}\subseteq P^{14}_8$ is a divisor on  $\cR_{8,\ell}$ for each other level $\ell\geq 3$.
The class of the compactification of $\cZ_{8,\ell}$ is given by the following formula, see \cite{chiodetal} Theorem F:
$$[\overline{\cZ}_{8, \ell}]=27\lambda-4\delta_0^{'}-\sum_{a=1}^{\lfloor \frac{\ell}{2}\rfloor} \frac{4(a^2-a\ell+\ell^2)}{\ell}\delta_0^{(a)}\in CH^1(\pr_{8,\ell}).$$
We refer to \cite{chiodetal} Section 1.4, for the definition of the boundary divisor classes $\delta_0^{(a)}$, where $a=1, \ldots, \lfloor \frac{\ell}{2}\rfloor$. If $\pi:\rr_{g,\ell}\rightarrow \mm_g$ is the map forgetting the level $\ell$ structure, then $$\pi^*(\delta_0)=\delta_0^{'}+\delta_0^{''}+\ell\sum_{\ell=1}^{\lfloor \frac{\ell}{2}\rfloor} \delta_0^{(a)}.$$}
\end{rema}

\vskip 5pt

We fix now a genus $8$ Prym-canonically embedded curve $\phi_L:C \hookrightarrow \PP^6$. As usual, we denote the kernel bundle by $M_L:=\Omega_{\PP^6|C}^1(1)$, hence we have the exact sequence
\begin{equation}
0\longrightarrow N_C^{\vee}\otimes L^{\otimes 4} \longrightarrow M_L\otimes L^{\otimes 3} \longrightarrow K_C\otimes L^{\otimes 4}\longrightarrow 0.
\end{equation}
This can be interpreted as an exact sequence of vector bundles over $\pr_8$. Denoting by $\H$ the vector bundle over $\pr_8$ with fibres $H^0(C,N_C^{\vee}\otimes L^{\otimes 4})$, we compute using the previous formulas
and the fact that $\mbox{rk}(\cN_1)=h^0(C,L)=7$ and $\mbox{rk}(\cN_3)=h^0(C,L^{\otimes 3})=35$:

\begin{equation}\label{chh}
c_1(\H)=35c_1(\cN_1)+7c_1(\cN_3)-c_1(\cN_4)-c_1(\G)=100\lambda-5\delta_0^{'}-\frac{53}{2}\delta_0^{\mathrm{ram}}.
\end{equation}


\vskip 4pt

Thus $\mathfrak{D}_1=\mathfrak{Kosz}_7\cap \cR_8$ and $\mathfrak{D}_2=\mathfrak{Kosz}_6\cap \cR_8$. We have already seen in Proposition \ref{proquarticsing} that $K_{1,2}(C,L)\neq 0$ if and only if  either $\phi_L(C)\subseteq \PP^6$ is not scheme-theoretically cut out by quadrics, or else, $H^1(\PP^6, \I_C^2(4))\neq 0$.
We write $$\cZ_8=\mathfrak{D}_1+\mathfrak{D}_2,\ \ \mbox{ where}$$
$$\mathfrak{D}_1:=\Bigl\{[C, \eta]\in \cR_8: \phi_L(C)\subseteq \PP^6 \mbox{ is scheme-theoretically not cut out by quadrics}\Bigr\}$$
and $$\mathfrak{D}_2:=\Bigl\{[C, \eta]\in \cR_8: H^1(\PP^6, \I_C^2(4))\neq 0\Bigr\}.$$

We have already observed that $\mbox{dim } I_{C,L}(2)=7$ and $\chi(\PP^6, \I_C^2(4))=28$. If $\cZ_8$ is a divisor, then $\mathfrak{D}_2$ is a divisor as well and for $[C, \eta]\in \cR_8\setminus \mathfrak{D}_2$,  we have that
$$\mbox{dim } \mbox{Sym}^2 I_{C,L}(2)= \mbox{ dim } I_{C,L}(4)_2=28.$$ Paying some attention to its definition, the divisor $\mathfrak{D}_1$ can be thought as the degeneracy locus
$$\Bigl\{[C, \eta]\in \cR_8: \mbox{Sym}^2 I_{C,L}(2)\stackrel{\neq}\longrightarrow I_{C,L}(4)_2\Bigr\},$$ which is an effective divisor on $\pr_8$.  We compute the class of this divisor:
\begin{theo}
We have the following formulas:
$$[\overline{\mathfrak{D}}_1]=7\lambda-\frac{1}{2}\delta_0^{'}-\frac{3}{4}\delta_0^{\mathrm{ram}}\in CH^1(\pr_8)$$
and
$$[\overline{\mathfrak{D}}_2]=20\lambda-\frac{7}{2}\delta_0^{'}-\frac{21}{4}\delta_0^{\mathrm{ram}}\in CH^1(\pr_8).$$
\end{theo}

\begin{proof}
We first globalize over $\pr_8$ the following exact sequence:
$$0 \longrightarrow I_{C,L}(4)_2\longrightarrow I_{C,L}(4) \longrightarrow H^0(C, N_C^{\vee}\otimes L^{\otimes 4})\longrightarrow H^1(\PP^6, \I_C^2(4))\longrightarrow 0.$$
Denote by $\cA$ the sheaf on $\pr_8$ supported along the divisor $\mathfrak{D}_2$, whose fibre over a general point of that divisor is equal to
to $H^1(\PP^6, \I_C^2(4))$. There is a surjective morphism of sheaves
$$\H\rightarrow \cA$$ and denote by $\G_4'$ its kernel. Since $\cA$ is locally free along $\mathfrak{D}_2$ and $\pr_8$ is a smooth stack, using the Auslander-Buchsbaum formula we find that $\G_4'$ is a locally free sheaf of rank equal to $\mbox{rk}(\H)=\chi(C, N_C^{\vee}(4L))=19\cdot 7$. Precisely, $\G_4'$ is an elementary transformation of $\H$ along the divisor $\mathfrak{D}_2$. Furthermore, $c_1(\G_4')=c_1(\H)-[\overline{\mathfrak{D}}_2]$.

\vskip 5pt

The morphism $\G_4\rightarrow \H$ globalizing the maps $I_{C,L}(4)\rightarrow H^0(C, N_C^{\vee}\otimes L^{\otimes 4})$ factors through the subsheaf $\G_4'$ and we form the exact sequence:
$$0\longrightarrow \G_4^2 \longrightarrow \G_4\longrightarrow \G_4'\longrightarrow 0.$$
The multiplication maps $\mbox{Sym}^2 I_{C,L}(2)\rightarrow I_{C,L}(4)_2$ globalize to a  sheaf morphism
$$\nu: \mbox{Sym}^2(\G_2)  \rightarrow \G_4^2$$ between locally free sheaves of the same rank $28$ over the stack $\pr_8$. The degeneration locus of $\nu$ is precisely the divisor
$\overline{\mathfrak{D}}_1$.
We compute:

$$c_1(\mbox{Sym}^2(\G_2))=8c_1(\G_2)=8(8c_1(\cN_1)-c_1(\cN_2))=-40\lambda+8(\delta_0^{'}+\delta_0^{\mathrm{ram}}),$$
and
$$c_1(\G^2_4)=120c_1(\cN_1)-c_1(\cN_4)-c_1(\H)+[\overline{\mathfrak{D}}_2]=-53\lambda+11\delta_0^{'}+\frac{25}{2}\delta_0^{\mathrm{ram}}+[\overline{\mathfrak{D}}_2].$$
We obtain the relation $[\overline{\mathfrak{D}_1}]-[\overline{\mathfrak{D}}_2]=-13\lambda+3\delta_0^{'}+\frac{9}{2}\delta_0^{\mathrm{ram}}.$
Since at the same time $$[\overline{\mathfrak{D}}_1]+[\overline{\mathfrak{D}}_2]=[\cZ_8]=27\lambda-4\delta_0^{'}-6\delta_0^{\mathrm{ram}},$$ we solve the system and conclude.
\hfill $\Box$
\end{proof}

We are now in a position to give a second proof of Theorem \ref{theomainintro}:

\begin{theo}\label{theomain}
The class $[\overline{\mathfrak{D}}_2]$ cannot be effective. It follows that $\cZ_8=\cR_8$ and $K_{1,2}(C, K_C\otimes \eta)\neq 0$, for every Prym curve
$[C, \eta]\in \cR_8$.
\end{theo}

\begin{proof}
We use the sweeping curve of the boundary divisor $\Delta_0^{'}$ of $\pr_8$ constructed via Nikulin surfaces in \cite{FV} Lemma 3.2: Precisely, through the general point of $\Delta_0^{'}$ there passes a rational curve $\Gamma\subseteq \Delta_0^{'}$, entirely contained in $\pr_8$, having the following numerical characters:
$$\Gamma \cdot \lambda=8, \ \ \Gamma\cdot \delta_0^{'}=42, \ \  \mbox{ and } \ \ \Gamma\cdot \delta_0^{\mathrm{ram}}=8.$$
We note that $\Gamma\cdot \overline{\mathfrak{D}}_2<0$. Writing $\overline{\mathcal{D}}_2\equiv \alpha\cdot \delta_0^{'}+E$, where $\alpha\geq 0$ and $E$ is an effective divisor whose support is disjoint from $\Delta_0^{'}$, we immediately obtain a contradiction.
\hfill $\Box$
\end{proof}

\vskip 3pt

The divisors $\mathfrak{D}_1$ and $\mathfrak{D}_2$ can be defined in an identical manner at the level of each moduli space $\rr_{8,\ell}$ of twisted level $\ell$ curves of genus $g$. As already pointed out, in the case $\ell\geq 3$ it follows from \cite{chiodetal} Proposition 4.4 that both $\mathfrak{D}_1$ and $\mathfrak{D}_2$ are actual divisors. Repeating the same calculations as for $\ell=2$, we obtain the following formula on the partial compactification $\pr_{8,\ell}$ of $\cR_{8,\ell}$:
\begin{equation}\label{d2}
[\overline{\mathfrak{D}}_2]=20\lambda-\frac{7}{2}\delta_0^{'}-\sum_{a=1}^{\lfloor \frac{\ell}{2}\rfloor}\frac{1}{2\ell}(7a^2-7a\ell+17\ell^2-20\ell)\delta_0^{(a)}\in CH^1(\pr_{8,\ell}).
\end{equation}

As an application, we mention a different proof of one of the main results from \cite{Br}:

\begin{theo}
The canonical class of $\rr_{8,\ell}$ is big for $\ell\geq 3$. It follows that  $\rr_{8,\ell}$ is a variety of general type for $\ell=3,4,6$.
\end{theo}
\begin{proof} Using formula (\ref{d2}), it is a routine exercise to check that for $\ell\geq 3$ the canonical class computed in \cite{chiodetal} Proposition 1.5
$$K_{\pr_{8,\ell}}=13\lambda-2\delta_0^{'}-(\ell+1)\sum_{a=1}^{\lfloor \frac{\ell}{2}\rfloor} \delta_0^{(a)}$$ can be written as a
\emph{positive} combination of the big class $\lambda$ and the effective class $[\overline{\mathfrak{D}}_2]$, hence it is big. Arguing along the lines of
\cite{chiodetal} Remark 3.5, it is easy to extend this result to the full compactification $\rr_{8,\ell}$ and deduce that $K_{\rr_{8,\ell}}$ is big.

To conclude that $\rr_{8,\ell}$ is of general type, one needs, apart from the bigness of the canonical class $K_{\pr_{8,\ell}}$ of the moduli stack, a result that the singularities of the coarse moduli space $\rr_{8,\ell}$ impose no adjunction conditions. This is only known for $2\leq \ell\leq 6, \ell\neq 5$, see \cite{CF}.
\hfill $\Box$
\end{proof}

\section{Rank $2$ vector bundles and  singular quartics}\label{secsecondproof}
Our goal in this section is to propose a construction of  syzygies of Prym canonical curves of genus $8$. We also sketch the  proof of the fact  that these syzygies are nontrivial.
We fix again a general element $[C,\eta]\in \cR_8$ and set $L:=K_C\otimes \eta$. According to Proposition \ref{proquarticsing}, in order
to prove that $K_{2,1}(C,L)\not=0$, we have to produce quartic hypersurfaces in $\PP^6$
which vanish at order at least $2$ along $\phi_L(C)$, but do not lie
in the image of the map  ${\rm Sym}^2I_{C,L}(2)\rightarrow I_{C,L}(4)$. The goal of this section
is to produce such quartics from rank $2$ vector bundles on $C$. The (incomplete) proof
that the quartics we construct are not in the image of ${\rm Sym}^2 I_{C,L}(2)$ depends on an unproved general position
statement ($\ast$), but there might be other approaches exploiting the fact that the hypersurfaces in question  are determinantal.

\vskip 4pt

The following construction produces quartics vanishing at order $2$ along $C$. Let $E$ be a rank $2$ vector bundle on $C$, with determinant $K_C$. Assume
\begin{eqnarray}\label{eqsectionsE} h^0(C, E)=4,\,\,\,h^0(C, E(\eta))=4.
\end{eqnarray}
Setting $V_0:=H^0(C,E)$ and $V_1:=H^0(C,E(\eta))$, we have a natural map
$$V_0\otimes V_1\rightarrow H^0(C,L),$$
defined using evaluation and the following composite map:
\begin{eqnarray}
\label{eqcompomap18mai}
H^0(E)\otimes H^0(E(\eta))\rightarrow H^0(E\otimes E(\eta))\cong H^0(\mathcal{E}nd \ E\otimes L)\stackrel{\mathrm{Tr}}{\longrightarrow}H^0(C,L).
\end{eqnarray}
This map gives dually a morphism
$$H^0(C,L)^{\vee}\rightarrow V_0^{\vee}\otimes V_1^{\vee},$$
(which will be proved below to be injective for a general choice of $E$).
We consider the quartic hypersurface $D_4$ on $\PP(V_0^{\vee}\otimes V_1^{\vee})$ parametrizing tensors of rank at most $3$.

\begin{lemm}\label{lesingquartic} The restriction $D_{4,E}$ of this quartic to $\PP\bigl(H^0(C,L)^{\vee}\bigr)\subseteq \PP\bigl(V_0^{\vee}\otimes V_1^{\vee}\bigr)$
is singular along the curve $C$.
 \end{lemm}
 \begin{proof} The quartic $D_4$ is singular along the set $T_2\subseteq \PP(V_0^{\vee}\otimes V_1^{\vee})$
of tensors of rank at most $2$.
The  quartic $D_{4,E}$ in $\PP(H^0(C,L)^{\vee})$ is thus singular
 along $T_2\cap \PP(H^0(C,L)^{\vee})$, which obviously contains
$C\subseteq \PP\bigl(H^0(C,L)^{\vee}\bigr)$, since at a point $p\in C$, the  map
$V_0\otimes V_1\rightarrow H^0(C,L)$ composed with the evaluation at $p$ factors through
$E_{\mid p}\otimes E(\eta)_{\mid p}$.
\hfill $\Box$
\end{proof}

By Brill-Noether theory, the variety $W^1_7(C)$ of degree $7$ pencils on $C$ is $4$-dimensional.
There should thus exist finitely many elements $D\in W^1_7(C)$ with the property that
\begin{eqnarray}\label{eqsecinv} h^0(C,D)\geq2,\,h^0(C, D\otimes \eta)\geq2.
\end{eqnarray}

We now have  the following lemma:
\begin{lemm}\label{leaajouter} Let $[C,\eta]\in \cR_8$ be as above and $D\in W^1_7(C)$ satisfying
(\ref{eqsecinv}). Then
\begin{itemize}
\item[\rm (i)] $ h^0(C, D)=2 \ \mbox{ and } h^0(C, D\otimes \eta)=2$. The multiplication map
$$\Bigl(H^0(C, D)\otimes H^0(C, K_C\otimes D^{\vee})\Bigr)\oplus \Bigl(H^0(C,D\otimes \eta)\otimes H^0(C,K_C\otimes D^{\vee}\otimes \eta )\Bigr)\rightarrow H^0(C,K_C)$$
is surjective (in fact, an isomorphism).
\item[\rm (ii)] The multiplication map
$$\Bigl(H^0(C, D)\otimes H^0(C,K_C\otimes D^{\vee}\otimes \eta)\Bigr) \oplus \Bigl(H^0(C,D\otimes \eta)\otimes H^0(C, K_C\otimes D^{\vee})\Bigr)\rightarrow H^0(C,K_C(\eta))$$
is surjective.
\end{itemize}
\end{lemm}
\begin{proof} This can be proved by a degeneration argument, for example by degenerating $C$ to the union of two curves of genus $4$ meeting at one point.
\hfill $\Box$
\end{proof}

By Brill-Noether theory, the following corollary follows from (i) above:
\begin{coro} \label{lefinite} For $[C,\eta]$ as above,
the set of pencils $D\in W^1_7(C)$  satisfying (\ref{eqsecinv}) is finite.
\end{coro}

Given such a $D$, we form the rank $2$ vector bundle
$$E=D\oplus (K_C\otimes D^{\vee})$$ on
$C$ which satisfies the conditions (\ref{eqsectionsE}). The associated quartic is however not
interesting for our purpose, due to the following fact:

\begin{lemm} \label{lequarticquadric} The quartic on $\PP(H^0(C,L)^{\vee})$
associated to the vector bundle $D\oplus (K_C\otimes D^{\vee})$ is the union of the two quadrics
$Q_0$ and $Q_1$ associated respectively with the multiplication maps
$$H^0(D)\otimes H^0((K_C\otimes D^{\vee})(\eta))\rightarrow H^0(K_C(\eta)) \mbox{ and  } H^0(D(\eta))\otimes H^0(K_C\otimes D^{\vee})\rightarrow H^0(K_C(\eta)).$$
Both these quadrics contain $C$.
\end{lemm}

\begin{proof} Indeed we have in this case
$$V_0=H^0(C,E)=H^0(C,D)\oplus H^0(C,K_C\otimes D^{\vee}), \ \mbox{respectively} $$
$$V_1=H^0(C,E(\eta))=H^0(C,D\otimes \eta)\oplus H^0(C,K_C\otimes D^{\vee}\otimes \eta).$$
Furthermore, it is clear that the map of (\ref{eqcompomap18mai}) factors through
the projection
$$V_0\otimes V_1\rightarrow \Bigl(H^0(C,D)\otimes H^0\bigl(C,K_C\otimes D^{\vee}\otimes \eta\bigl)\Bigr)\oplus \Bigl(H^0(C,K_C\otimes D^{\vee})\otimes H^0(C,D\otimes \eta)\Bigr)$$
and induces on each summand the multiplication map.
The quadric $Q_0$ is by definition associated with the
the multiplication map $$\mu_0: H^0(C,D)\otimes H^0(C,K_C\otimes D^{\vee}\otimes \eta)\rightarrow H^0(C,K_C\otimes \eta),$$ and
is the set of elements $f$ in $\PP(H^0(K_C\otimes \eta))^{\vee}$ such that $\mu_0^*(f)$ is a tensor of rank
$\leq 1$. Similarly for $Q_1$, with $D$ being replaced with $D(\eta)$.
Finally we use the fact that  a
tensor $$(\mu_0^*f,\mu_1^*f)\in \Bigl(H^0(C, D)\otimes H^0(C,K_C\otimes D^{\vee}\otimes \eta)\Bigr)\oplus \Bigl(H^0(C,K_C\otimes D^{\vee})\otimes H^0(C, D\otimes \eta)\Bigr)$$
has rank at most $3$ if and only if one of  $\mu_0^*f$ and $\mu_1^*f$ has rank at most $1$.
\hfill $\Box$
\end{proof}

We recall from \cite{mumford} or \cite{mukai} that the Brill-Noether condition
$h^0(C, E)\geq 4$ imposes only $10={5\choose 2}$ equations on the parameter space of rank $2$ vector bundles $E$ with determinant $K_C$.
As $\mbox{det } E(\eta)\cong K_C$ as well, we conclude that the equations
(\ref{eqsectionsE}) impose only $20$ conditions. As the moduli space $\mathcal{SU}_C(2,K_C)$ of semistable rank $2$ vector bundles on $C$ having determinant $K_C$ has dimension $3g-3=21$, in our case we conclude that
there is a positive dimensional family of such vector bundles on $C$ satisfying (\ref{eqsectionsE}).

\vskip 4pt

We now sketch the proof of the fact that
 for $C$  general of genus $8$ and $D\in W^1_7(C)$ satisfying (\ref{eqsecinv}), for a general deformation $E$ of  the vector bundle $D\oplus (K_C\otimes D^{\vee})$ satisfying
${\rm det}\,E\cong K_C$ and $h^0(C, E)=4$, the associated quartic $D_{4,E}$ singular along $C$ is not defined by
an element of ${\rm Sym}^2 I_C(2)$.
 Combined with Proposition \ref{proquarticsing}, this provides a  third  approach to Theorem \ref{theomainintro}. The proof of this fact rests on
an unproven general position  statement ($\ast$), so it is incomplete.

\begin{proof}[Sketch of proof of the nontriviality of the syzygy] The vector bundle $E$ is generated by sections, as it is a general section-preserving deformation
 of the vector bundle
$$D\oplus (K_C\otimes D^{\vee})$$  which is generated by global sections, and similarly for $E(\eta)$.
Along $C\subseteq \PP(H^0(C,L)^{\vee})$,  then the rational map
$$\PP\bigl(H^0(C,L)^{\vee}\bigr)\dashrightarrow \PP\bigl(H^0(E)^{\vee}\otimes H^0(E(\eta))^{\vee}\bigr)$$ is well-defined and the image
of $C$ is contained in the locus $T_{2,E}$ of tensors of rank exactly $2$.
 In fact, the case of $D\oplus (K_C\otimes D^{\vee})$ shows that this map is a morphism for general $E$
 (one just needs to know that $H^0(C,K_C\otimes \eta)$ is generated by the two vector spaces
 $H^0(D)\otimes H^0(K_C\otimes D^{\vee}\otimes \eta)$ and $H^0(D\otimes \eta)\otimes H^0(K_C\otimes D^{\vee})$ respectively, or rather their images under the multiplication map. Note that on $T_{2,E}$, there is a rank $2$ vector bundle $M$ which restricts
 to $E$ on $C$.

\vskip 4pt

In the case of the split vector bundle $E_{\mathrm{sp}}=D\oplus (K_C\otimes D^{\vee})$, Lemma \ref{lequarticquadric} shows that the Zariski closure
$\overline{T_{2,E_{\mathrm{sp}}}}$ parameterizing tensors of rank $\leq 2$ in
$\PP(H^0(C,L)^{\vee})\subseteq \PP(V_0^{\vee}\otimes V_1^{\vee})$  is equal to the singular locus
of $D_{4,E_{\mathrm{sp}}}$ and consists of the union of the two planes $P_0,\,P_1$ defined as  the singular loci of the quadrics
$Q_0,\,Q_1$ respectively, and  the intersection $Q_0\cap Q_1$. The
 locus $\overline{T_{2,E_{\mathrm{sp}}}}\setminus T_{2,E_{\mathrm{sp}}}$ is the locus where
 the tensor has rank $1$, and this happens exactly along the two conics
 $P_0\cap Q_1$ and $P_1\cap Q_0$. The curve $C$ is contained in $Q_0\cap Q_1$
and does not intersect $P_0\cup P_1$. In particular, the rational map
$\phi:\widetilde{\PP^6}\dashrightarrow\PP^6$ given by the linear system
$I_C(2)$ is well defined along $P_0\cup P_1$.
We believe that the  following general position statement  concerning  the two planes $P_i$ is true for general $C$  and $D,\,\eta$  as above.

\vspace{0,5cm}

($\ast$) {\it  The
 surfaces $ \phi(P_i)$ are projectively normal Veronese surfaces, generating
a hyperplane $\langle \phi(P_i)\rangle\subseteq \PP^6$. Furthermore,
the surface $\phi(P_0)\cup\phi(P_1)\subseteq \PP^6$ is contained in a unique quadric in
$\PP^6$, namely the union of the two hyperplanes $\langle \phi(P_0)\rangle $ and $\langle \phi(P_1)\rangle$.}

\vspace{0,5cm}

We now prove that, assuming ($\ast$), for a general vector bundle $E$ as above,  the associated quartic $D_{4,E}$ singular along $C$ is not defined by
an element of ${\rm Sym}^2 I_C(2)$.
As $P_0,\,P_1$ are $2$-dimensional reduced components of $\overline{T_{2,E_{\mathrm{sp}}}}$, hence of the right dimension, the theory
of determinantal hypersurfaces shows that for general
$E$ as above, there is a reduced surface $\Sigma_E\subseteq \overline{T_{2,E}}$ whose specialization
when $E=E_{\mathrm{sp}}$ contains $P_0\cup P_1$. Let $\mathcal{E}\rightarrow C\times B$ be a family of vector bundles on $C$ parameterized by a smooth curve
$B$, with general fiber $E$ and special fiber $E_{\mathrm{sp}}$. Denote by $\mathcal{E}_b$ the restriction of $\mathcal{E}$ to $C\times \{b\}$.
 Property ($\ast$) then implies that $\phi(\Sigma_{\mathcal{E}_b})$ for general $b\in B$ is contained in at most one quadric
 $Q_{\mathcal{E}_b}$ in $\PP^6$. We argue by contradiction and assume that the quartic  $D_{4,\mathcal{E}_b}$ is a pull-back $\phi^{-1}(Q)$ for general $b$. One thus must
have $Q=Q_{\mathcal{E}_b}$. Next, the determinantal quartic $D_{4,\mathcal{E}_b}$ is singular along $T_{2,\mathcal{E}_b}$, hence along $\Sigma_{\mathcal{E}_b}$.
Let $b\mapsto q_{\mathcal{E}_b}\in {\rm Sym}^2I_C(2)$, where $q_{\mathcal{E}_b}$ is a defining equation for
the quadric  $Q_{\mathcal{E}_b}$. Then we find that the first order derivative of the family
$\phi^*q_{\mathcal{E}_b}$ at $b_0$ also vanishes along $\Sigma_{\mathcal{E}_{b_0}}$, hence it must be proportional to $\phi^*q_{\mathcal{E}_{b_0}}$. We then conclude that the quadric  $Q_{\mathcal{E}_b}$
is in fact constant, and thus must be equal to the quadric $Q_{E_{\mathrm{sp}}}$.
We now reach a contradiction by proving the following lemma.
\hfill $\Box$
\end{proof}

\begin{lemm}\label{lefinfin} If the determinantal quartic $D_{4,\mathcal{E}_b}$ is constant, equal to
$D_{\mathrm{sp}}=Q_0\cup Q_1$, then the vector bundle $\mathcal{E}_b$ on $C$ does not deform with $b\in B$.
\end{lemm}
\begin{proof}
Denoting $V_{0,b}:=H^0(C,\mathcal{E}_b)$, $V_{1,b}:=H^0(C,\mathcal{E}_b(\eta))$,
we have the multiplication map
$$V_{0,b}\otimes V_{1,b}\rightarrow H^0(C,K_C\otimes \eta)$$
which is surjective for generic $b$ since it is surjective for $\mathcal{E}_0=D\oplus (K_C\otimes D^{\vee})$
(see Lemma \ref{leaajouter}). The determinantal quartic $D_{4,\mathcal{E}_b}$ is the vanishing locus of the determinant of the
corresponding bundle map
\begin{eqnarray}\label{eqbundlemap}\sigma_b: V_{0,b}\otimes \mathcal{O}_{\PP(H^0(C,K_C(\eta))^{\vee})}\rightarrow V_{1,b}^{\vee}\otimes \mathcal{O}_{\PP(H^0(C,K_C(\eta))^{\vee})}(1)
\end{eqnarray}
on $ \PP\bigl(H^0(C,K_C\otimes \eta)^{\vee}\bigr)$.
We know that $D_{4,\mathcal{E}_b}=Q_0\cup Q_1$ for any $b\in B$, where the quadrics $Q_i$ are singular (of rank $4$), but with singular locus
$P_i$ not intersecting $C\subseteq Q_0\cap Q_1$.
The morphism $\sigma_b$ has rank exactly $1$ generically along each $Q_i$ and the kernel of $\sigma_{\mid D_{4,b}}$
determines a line bundle $\mathcal{K}_{i,b}$ on its smooth locus $Q_i\setminus P_i$. This line bundle is independent of $b$ since ${\rm Pic}(Q_i\setminus P_i)$ has no continuous part.
The restriction of $\mathcal{K}_{i,b}$ to $C$ is thus constant. Finally, on the smooth part of $(Q_0\cap Q_1)_{\mathrm{reg}}$, the kernel ${\rm Ker}(\sigma)$ contains the two line bundles
$\mathcal{K}_{i,b\mid Q_0\cap Q_1}$. Restricting to $C\subseteq (Q_0\cap Q_1)_{\mathrm{reg}}$, we conclude that
${\rm Ker}\,\sigma_{b\mid C}$ contains $\mathcal{K}_{i,0\mid C}$ for $i=0,1$. For $b=0$, one has $${\rm Ker}\,\sigma_{0\mid C}=\mathcal{K}_{0,0\mid C}\oplus \mathcal{K}_{1,0\mid C}$$
and this thus remains true for general $b$. Finally, it follows from the construction and the fact that
$\mathcal{E}_b$ is generated by its sections that ${\rm Ker}\,\sigma_{b\mid C}=\mathcal{E}_b^{\vee}$, which finishes the proof.
\hfill $\Box$
\end{proof}

\section{Miscellany}
\subsection{Extra remarks on  the geometry of paracanonical curves of genus $8$ with a nontrivial syzygy}

We now comment on an interesting rank $2$ vector bundle appearing in our situation.
Again, let $\phi_L:C\hookrightarrow \PP^6$ be a paracanonical curve
of genus $8$. We assume $L$ is scheme-theoretically cut out by quadrics. Denoting by
$N_C$  the normal bundle
of $C$ in the embedding in $\PP^6$, we consider the natural map
$I_C(2)\otimes\mathcal{O}_C
\rightarrow N_C^{\vee}\otimes L^{\otimes 2}$ (which is surjective by our assumption)
given by differentiation along $\phi_L(C)$, and let $F$ denote its kernel. We thus have
the short exact sequence:
\begin{eqnarray}\label{eqev} 0\longrightarrow F\longrightarrow I_C(2)\otimes \mathcal{O}_C
\longrightarrow N_C^{\vee}\otimes L^{\otimes 2}\longrightarrow 0.
\end{eqnarray}
If $K_{1,2}(C,L)\not=0$, the map $\mu:I_C(2)\otimes H^0( \PP^6,\mathcal{O}(1))
\rightarrow I_C(3)$ is not surjective, hence not injective.
A fortiori, the map
$$\overline{\mu}:I_C(2)\otimes H^0( \PP^6,\mathcal{O}_{\PP^6}(1))
\rightarrow H^0(C,N_C^{\vee}\otimes L^{\otimes 3})$$
induced by (\ref{eqev}) is not injective, so that
$h^0(C,F(L))\not=0$. In fact, the equivalence between the statements
$h^0(C,F(L))\not=0$ and $K_{1,2}(C,L)\not=0$ follows from the same argument
once we know that there is no cubic polynomial on $\PP^6$ vanishing with multiplicity
$2$ along $C$.

\vskip 4pt

We observe now that $F$ is a vector bundle of rank $2$ on the curve $C$, with determinant
equal to ${\rm det}\,N_C\otimes L^{\otimes (-2)}\cong K_C\otimes L^{\otimes (-3)}$.
Hence if $F(L)$ has a nonzero section,
assuming this section vanishes nowhere along $C$,
then $F(L)$ is an extension of $K_C\otimes L^{\vee}$ by $\mathcal{O}_C$. This provides an extension class
\begin{eqnarray}\label{eqextension} e\in H^1(C,L\otimes K_C^{\vee})=H^0(C, K_C^{\otimes 2}\otimes L^{\vee})^{\vee}.
\end{eqnarray}
Assume now $L\otimes K_C^{\vee}=:\eta$ is a nonzero $2$-torsion element of ${\rm Pic}^0(C)$.
Then $$e \in H^0(C,L)^{\vee}.$$
On the other hand, according to Theorem \ref{theomain}, there exists
a nontrivial syzygy
$$\gamma=\sum_{i=1}^6 \ell_i\otimes q_i\in K_{1,2}(C,L)={\rm Ker}\bigl\{H^0(\PP^6, \mathcal{O}_{\PP^6}(1))\otimes I_C(2) \rightarrow I_C(3)\bigr\},$$
which is degenerate by Proposition \ref{prodeg}. As we saw already, it has in fact rank $6$
for generic $[C,\eta]$, hence determines a nonzero element
\begin{eqnarray}\label{eqpourf}f\in H^0(\PP^6, \mathcal{O}_{\PP^6}(1))^{\vee}=H^0(C,L)^{\vee}=H^1(C,K_C\otimes L^{\vee})=H^1(C,L\otimes K_C^{\vee}),
\end{eqnarray}
which is well-defined up to a coefficient.

\begin{prop} \label{proeandf} The two elements $e$ and $f$ are  proportional.
\end{prop}
\begin{proof}  Equivalently, we show that the kernels
of the two linear forms  $e,\,f\in H^0(C,L)^{\vee}$ are equal.
Viewing $\gamma$ as an element of ${\rm Hom}\,(I_C(2)^{\vee},H^0(C,L))$, we have
${\rm Ker}(f)={\rm Im}(\gamma)$. On the other hand,
the kernel of $e$ identifies with
$${\rm Im}\,\Bigl\{j:H^0(C,F\otimes L^{\otimes 3}\otimes K_C^{\vee})\rightarrow H^0(C,L)\Bigr\},$$ where the map $j$ is obtained by twisting the
exact sequence
$0\rightarrow\mathcal{O}_C\rightarrow F(L)\rightarrow K_C\otimes L^{\vee}\rightarrow 0$ by
$K_C$.
We have $F\otimes L^{\otimes 3}\otimes K_C^{\vee}\cong F^{\vee}$ since ${\rm det}\,F\cong K_C\otimes L^{\otimes (-3)}$, hence there is a natural morphism
$$i^*: I_C(2)^{\vee}\otimes \mathcal{O}_C\rightarrow F^{\vee}\cong F(L^{\otimes 3}\otimes K_C^{\vee})$$
dual to the inclusion $F\hookrightarrow I_C(2)\otimes \mathcal{O}_C$ of (\ref{eqev}).
The proposition follows from the following claim:

\vskip 3pt

\noindent {\bf Claim.} {\it The morphism $\alpha:I_C(2)^{\vee}\rightarrow H^0(C,L)$ is equal to
$j\circ i^*$.}

\vskip 3pt

Forgetting about the last identification $F^{\vee}\cong F\otimes L^{\otimes 3}\otimes K_C^{\vee})$, the claim amounts to the following
 general fact:
For an evaluation exact sequence on a variety $X$
$$0\longrightarrow G\longrightarrow W\otimes \mathcal{O}_X\longrightarrow M\longrightarrow 0$$
and for a section
$s\in H^0(X,G(L))=H^0(X,\mathcal{H}om(G^{\vee},L))$
giving an element $$s'\in {\rm Ker}\Bigl\{W\otimes H^0(X,L)\rightarrow H^0(X,M\otimes L)\Bigr\}\subseteq {\rm Hom}\,(W^{\vee},H^0(X,L)),$$
the induced map $s:H^0(X,G^{\vee})\rightarrow H^0(X,L)$ composed with the map
$W^{\vee}\rightarrow H^0(X,G^{\vee})$ equals the map $s':W^{\vee}\rightarrow H^0(X,L)$.
\hfill $\Box$
\end{proof}

\subsection{Further properties}

Using the exact sequence (\ref{eqev}) in the general case of a genus 8 paracanonical curve $[C,L]\in P^{14}_8$, we obtain:
\begin{lemm}
 A section $s\in H^0(C,F(L))\subseteq I_{C,L}(2)\otimes H^0(C,L)=\mathrm{Hom}\bigl(I_{C,L}(2)^{\vee},H^0(C,L)\bigr)$  of rank 6, determines an element $e\in |2L-K_C|$.
\end{lemm}
\begin{proof}
The multiplication by $s\in H^0(F(L))\subseteq I_{C,L}(2)\otimes H^0(C,L)=H^0(I_{C,L}(2)^{\vee}\otimes L)$ determines the natural maps $ F^{\vee}\to L$  and $g_{s}:I_C(2)^{\vee}\otimes \OO_C\to  L$  sitting in the following diagram:
$$
\begin{array}{ccccccccc}
0&\longrightarrow &\mathcal{K}er (g_s) &\longrightarrow& I_C(2)^{\vee}\otimes \OO_C&\longrightarrow &L&\longrightarrow &0\\
&&\downarrow&&\downarrow&&\downarrow&&\\
0&\longrightarrow  &2L-K_C&\longrightarrow&F^{\vee}&\longrightarrow &L &\longrightarrow &0\\
\end{array},
$$
where $I_C(2)^{\vee}\otimes \OO_C\rightarrow F^{\vee}$ is the dual of the natural inclusion of (\ref {eqev}). Passing to global sections we get the inclusion  $H^0(\mathcal{K}er(g_s))=\mbox{Ker}\bigl\{I_{C,L}(2)^{\vee}\to H^0(C, L)\bigr\}\hookrightarrow H^0(2L-K_C)$, which by hypothesis in $1$-dimensional hence it defines an element $e\in |2L-K_C|$.
\hfill $\Box$
\end{proof}

Via the exact sequence (\ref{eqev}) we can also show directly the following result that has been used
in Section \hyperref[secthirdproof]{3}:

\begin{lemm}\label{ID} If there is a spin curve $D=C\cup E\hookrightarrow \PP^6$ of genus 22 and degree 21 containing the genus 8 paracanonical curve $[C,L]$ as in Lemma \ref{DCE} , then $H^0C, (F(L))\neq 0$. If there is no cubic polynomial on $\PP^6$ vanishing with multiplicity
$2$ along $C$, then $K_{1,2}(C,L)\neq 0$.
\end{lemm}

\begin{proof} Let $e=C\cap E$ and recall $c_1(F)=-3L+K_C$ and $\OO_C(e)=2L-K_C$. Note that $I_D(2)\subseteq I_C(2)$ is 6-dimensional. Tensor then  the first vertical exact sequence of the following diagram by $L$ and pass to global sections.
$$
\begin{array}{ccccccccc}
&&0&&0& &0&&\\
&&\downarrow&&\downarrow&&\downarrow&&\\
0&\longrightarrow &L^{\vee}&\longrightarrow& I_D(2)\otimes \mathcal{O}_C&\longrightarrow &\mathcal I_D/(\mathcal I_D\cap \mathcal I_C^2)(2)&\longrightarrow &0\\
&&\downarrow&&\downarrow&&\downarrow&&\\
0&\longrightarrow &F&\longrightarrow& I_C(2)\otimes \mathcal{O}_C&\longrightarrow &N_C^{\vee}(2)&\longrightarrow &0\\
&&\downarrow&&\downarrow&&\downarrow&&\\
0&\longrightarrow &\OO_C(-e)&\longrightarrow& \mathcal{O}_C&\longrightarrow &\OO_{C|e}&\longrightarrow &0\\
&&\downarrow&&\downarrow&&\downarrow&&\\
&&0&&0& &0&&\\
\end{array}.
$$
\hfill $\Box$
\end{proof}

\subsection{Nontrivial syzygies of paracanonical curves via vector bundles}
We return to the proof of Theorem \ref{theomain} given in Section \hyperref[secsecondproof]{5}.
Consider now a general paracanonical curve $[C,K_C\otimes \eta]\in P^{14}_8$. For a rank $2$ vector bundle on $C$ of degree $14$, with noncanonical determinant,
the equation $h^0(C, E)\geq 4$ imposes $16$ conditions. Similarly, if $\epsilon \in \mbox{Pic}^0(C)$, the equation $h^0(C, E\otimes \epsilon)\geq 4$ imposes $16$ conditions on the parameter space of $E$. Given $C$, there are
$29=4g-3$ parameters for $E$, and $8=g$ parameters for $\epsilon$. It follows that we have at least a $5$-dimensional family
of pairs $(E,\epsilon)$, such that
\begin{eqnarray}\label{eqEpara} h^0(C,E)\geq4 \ \mbox{ and } \ h^0(C,E\otimes \epsilon)\geq 4.
\end{eqnarray}
Furthermore, the construction of Section \hyperref[secsecondproof]{5} (together with Proposition
\ref{proquarticsing}) shows that for a general triple $(C,E,\epsilon)$ as above, one has
$K_{2,1}(C,L)\not=0$, where $L:={\rm det}\,E\otimes \epsilon$.
Assuming the map $(E,\epsilon)\mapsto L$ is generically finite on its image, we constructed in this way
a five dimensional family of paracanonical line bundles $L\in \mbox{Pic}^{14}(C)$  with a nontrivial syzygy: $K_{1,2}(C,L)\not=0$.
This family has the following property:

\begin{lemm} If $L={\rm det}\,E\otimes \epsilon$, where $E$ satisfies (\ref{eqEpara}), the line bundle
$K_C^{\otimes 2}\otimes L^{\vee}$ satisfies the same property. The  family above, which has dimension at least five, is thus invariant under
the involution $L\mapsto K_C^{\otimes 2}\otimes L^{\vee}$ on $P^{14}_8$, whose fixed locus is the Prym moduli space $\cR_8$.
\end{lemm}
\begin{proof} This follows from Serre duality, replacing $E$ with
$E^{\vee}\otimes K_C$ and  $E\otimes \epsilon$ by $E^{\vee}\otimes \epsilon^{\vee}\otimes K_C$ plus the fact that
 ${\rm det}\,(E^{\vee}\otimes K_C)\otimes \epsilon^{\vee}\cong K_C^{\otimes 2}\otimes {\rm det}\,E^{\vee}\otimes \epsilon^{\vee}$.
\hfill $\Box$
\end{proof}

One can ask in general the following question:

\begin{question} Is the divisor $\mathfrak{Kosz}$ on $P^{14}_8$  invariant under the involution $L\mapsto K_C^{\otimes 2}\otimes L^{\vee}$?
 \end{question}

\providecommand{\bysame}{\leavevmode\hbox to3em{\hrulefill}\thinspace}
%
%

\bibliographystyle{amsalpha}
\bibliographymark{References}

\end{document}